\documentclass{tran-l}

\usepackage{amssymb}
\usepackage{amsmath}

\usepackage{diagrams}

\newtheorem{theorem}{Theorem}[section]

\newtheorem{corollary}[theorem]{Corollary}

\newtheorem{definition}[theorem]{Definition}

\newtheorem{lemma}[theorem]{Lemma}

\newtheorem{proposition}[theorem]{Proposition}

\newarrow {Partial} ----{harpoon}
\newarrow{Into} C--->

\begin{document}

\title{Representing geometric morphisms using power locale monads}
\author{Christopher F. Townsend}
\address{8 Aylesbury Road, Tring, Hertfordshire, HP23 4DJ, U.K.}

\maketitle

\begin{abstract}
It it shown that geometric morphisms between elementary toposes can be
represented as adjunctions between the corresponding categories of locales.
The adjunctions are characterized as those that preserve the order
enrichment, commute with the double power locale
monad and whose right adjoints preserve finite coproduct. They are also characterized as those adjunctions that
preserve the order enrichment and commute with both the lower and the upper power locale monads.
\end{abstract}

\section{Introduction}

Every geometric morphism $f:\mathcal{F\rTo E}$ between elementary
toposes gives rise to an order-enriched adjunction $\Sigma _{f}\dashv f^{\ast }$ between
the category of locales in $\mathcal{F}$ and the category of locales in $%
\mathcal{E}$ with the right adjoint being given by pullback, $f^{\ast }:%
\mathbf{Loc}_{\mathcal{E}}\rTo \mathbf{Loc}_{\mathcal{F}}$. This
pullback functor has the known properties that it preserves order
enrichment, commutes with finite coproducts and there is a monad isomorphism
\begin{equation*}
f^{\ast }\mathbb{P}_{\mathcal{E}}\cong \mathbb{P}_{\mathcal{F}}f^{\ast }\text{,}
\end{equation*}
where $\mathbb{P}$ is the double power locale construction. The main aim of this paper is to show that any order-enriched
adjunction between locales in $\mathcal{F}$ and locales in $\mathcal{E}$
that satisfies these properties arises as the pullback adjunction of a
geometric morphism unique up to natural isomorphism.

This provides a representation theorem for geometric morphisms showing a new relationship between the power locale monads and the morphisms of topos theory. In some categorical approaches to locale theory (e.g. \cite{towhofman} and \cite{Vicpoints}) the power locale monads take on roles analogous to the more familiar power set monad of elementary topos theory. It is the power locale monads that provide the structure of locale theory in a manner that is similar to the way that the powerset provides the structure of set theory. The implication of the representation theorem of this paper is therefore that in topos theory the usual notion of morphism (i.e. geometric morphism) can in fact be interpreted as structure preserving map.

\section{The main theorem and summary proof}

The reader is assumed to be familiar with locale theory and topos theory, \cite{Elephant}. The main theorem that is to be proved is now stated, but note that the detailed definitions of the terms that are used will be given in the main body of the paper.
\begin{theorem}
\label{maintheorm}
For any two toposes $\mathcal{E}$ and $\mathcal{F}$ there are categorical equivalences between\\
(i) the category of geometric morphisms from $\mathcal{F}$ to $\mathcal{E}$,  \\
(ii) the category of order-enriched adjunctions $L\dashv R:
\mathbf{Loc}_{\mathcal{F}}\pile{\rTo \\ \lTo} \mathbf{Loc}_{\mathcal{E}}$ with $R$ preserving finitary coproduct and for which there exists some monad isomorphism $\phi:R\mathbb{P}_{\mathcal{E}}\cong \mathbb{P}_{\mathcal{F}}R$ such that $R$ preserves, up to the monad isomorphism $\phi$, the strength on $\mathbb{P}$,\\
(iii) the category of order-enriched adjunctions $L\dashv R:
\mathbf{Loc}_{\mathcal{F}}\pile{\rTo \\ \lTo} \mathbf{Loc}_{\mathcal{E}}$ with $R$ preserving finitary coproduct and for which there exists some monad isomorphism $\phi:R\mathbb{P}_{\mathcal{E}}\cong \mathbb{P}_{\mathcal{F}}R$; and, \\
(iv) the category of order-enriched adjunctions $L\dashv R:
\mathbf{Loc}_{\mathcal{F}}\pile{\rTo \\ \lTo} \mathbf{Loc}_{\mathcal{E}}$ for which there exists two monad isomorphisms $\phi_L:R P_L^{\mathcal{E}}\cong P_L^{\mathcal{F}}R$ and $\phi_U:R P_U^{\mathcal{E}}\cong P_U^{\mathcal{F}}R$ such that $R$ preserves, up to these monad isomorphisms, the canonical distribution isomorphism $P_L P_U \cong P_U P_L$.
\\
\end{theorem}
The morphisms of the categories (ii) to (iv) are natural transformations between the right adjoints and are not required to interact with the monad isomorphisms.

Let us give an overview of how the proof is going to work. In each part of the proof what is required is a check that the various conditions placed on an adjunction $L \dashv R$ are equivalent. To do this we will repeatedly rely on a characterization of monad isomorphisms which occurs when there is an adjunction between categories, each with a monad. This categorical result is that given a functor (with a left adjoint) a monad isomorphism can be constructed if and only if there is a lifting of the adjunction to an adjunction between the respectively Kleisli categories. The result, which establishes a bijection between monad isomorphisms and liftings of adjunctions, can be derived from the more familiar fact that liftings of functors to Kleisli categories correspond to monad opfunctors.

Once this categorical lemma is established the proof focuses on the equivalence of (i) and (ii). Certainly any geometric morphism $f$ gives rise to a pullback adjunction $\Sigma_f \dashv f^*:\mathbf{Loc}_{\mathcal{F}}\pile{\rTo \\ \lTo} \mathbf{Loc}_{\mathcal{E}}$, see \cite{Elephant} where the notation $f_!$ is used rather than $\Sigma_f$. The first part of the proof concerns itself with showing that $\Sigma_f \dashv f^*$ satisfies the conditions of (ii). In fact it is well known that $f^*$ preserves finitary coproduct and commutes with the double locale monad so what is new is a verification that the strength is preserved. This is done by giving an explicit description of the monad isomorphism $\phi$ (equivalently the Kleisli lifting of $\Sigma_f \dashv f^*$) which witnesses that $f^*$ commutes with the double power monad. This explicit description exploits the fact that $\Sigma_f \dashv f^*$ satisfies Frobenius reciprocity (\cite{towgeom}) and so allows us to construct a lifting to Kleisli categories, the details of which are in fact shown in \cite{towgeom}. With this explicit description, preservation of the strength is an immediate application of naturality since it can be seen that preservation of the strength can be seen as preservation of an exponential in $[\mathbf{Loc}^{op},\mathbf{Set}]$.

To complete the proof of the equivalence of the categories defined by (i) and (ii) the techniques deployed are those of \cite{towgeom}. However here the situation appears necessarily more complex as in the absence of Frobenius reciprocity we are not able to be as explicit as we would like about the behaviour of the lifting to Kleisli categories that is implied by the monad isomorphism. We proceed as follows: given an order-enriched adjunction between categories of locales and a double power locale monad isomorphism (that also preserves the strength, and has $R$ preserving finitary coproduct) we must construct a geometric morphism. The existence of the monad isomorphism implies that there is a lifting of the order-enriched adjunction to Kleisli categories. This lifting shows how to extend both $L^{op}$ and $R^{op}$ from frame homomorphisms to dcpo homomorphisms. Since the property of being discrete can be characterized in terms of open maps and open maps can be characterized in terms of certain dcpo homomorphisms it follows that the right adjoint preserves discrete locales. Because any topos embeds in its category of locales as the full subcategory of discrete locales we have a candidate for the inverse image of a geometric morphism. Defining its right adjoint (i.e. the direct image) hinges on the observation that every object in a topos can be described canonically as the equalizer of a pair of dcpo homomorphisms between frames. Since $L^{op}$ also extends to dcpo homomorphisms the required left adjoint can be defined as the equalizer of the image of this canonical equalizer. Then, by application of the assumption that $R$ preserves the strength, it is possible to check that a geometric morphism has been defined whose pullback adjunction is (isomorphic to) the original order-enriched adjunction. The shape of this representation theorem has already appeared in \cite{towgeom}.

Certainly the conditions of (ii) imply those of (iii) as (iii) is weaker. To prove that the conditions of (iii) imply those of (ii) we need to check that the exponentials $\alpha^X$ in $[\mathbf{Loc}^{op},\mathbf{Set}]$ are preserved for any natural transformation $\alpha$ corresponding to a Kleisli morphism of the double power locale monad (i.e. corresponding to a dcpo homomorphism between frames). It can be seen that the relevant exponentials can be described by change of base functors extended to Kleisli categories (\cite{towdcpo}) and so the proof becomes about showing that $R$ commutes with change of base functors when lifted to Kleisli categories. This is done by providing an external description of the dcpo homomorphisms in each slice (i.e. dcpo homomorphisms that exist after change of base to $Sh(Y)$, the topos of sheaves over $Y$ for any locale $Y$). The external description is given in terms of weak triquotient assignments and quite a bit of background material on these maps is required to proceed.

The proof of (iii) implies (iv) follows from the definitions of the lower and upper power locale monads ($P_L$ and $P_U$). Their Kleisli categories have as morphisms reversed suplattice and preframe homomorphisms respectively. By (iii) both $L^{op}$ and $R^{op}$ lift to functors that preserve dcpo homomorphisms and so it needs to be verified that these liftings also preserves the property of being a join or meet semilattice homomorphism. For $L^{op}$ this is immediate as $L$ preserves coproduct as it is a left adjoint. $R$ on the other hand preserves finitary coproduct by assumption. The liftings that witness that $R$ commutes with $\mathbb{P}$ therefore also serve to witness that $R$ commutes with both $P_L$ and $P_U$. Checking that the canonical distribution is also preserved is then a matter of unwinding the definition of the distribution in terms of universal suplattice and universal preframe homomorphisms. Since both classes of universal maps are preserved by the liftings it follows that the distribution is preserved. The proof in the other direction essentially follows from the definitions since we have $P_L P_U \cong \mathbb{P} $. The proof requires a certain amount of diagram chasing. To prove that $R$ preserves finitary coproduct in effect we recall that $P_L$ takes binary coproduct to product (preserved by $R$) and further that $P_L$ reflects isomorphisms.

\section{Functors between categories, each with an order-enriched monad}

We start by proving a general categorical result (Lemma \ref{adjunction lemma}) which will be used repeatedly in the paper as a way of characterizing when monad isomorphisms exist.

A monad on an order-enriched category is said to be \emph{order-enriched} if its functor part is. If $ \mathcal{C}$ and $\mathcal{D}$ are two order-enriched categories, each with a monad, $\mathbb{T}_{C}= (T_{\mathcal{C}},\eta ^{\mathcal{C}},\mu ^{\mathcal{C}})$ and $ \mathbb{T}_{D} = (T_{\mathcal{D}},\eta ^{\mathcal{D}},\mu ^{\mathcal{D}})$, then a \emph{monad opfunctor} from $(\mathcal{C},\mathbb{T}_{C}) $ to $ (\mathcal{D},\mathbb{T}_{D})$ is a pair $(F,\phi)$ where $F:\mathcal{C} \rTo \mathcal{D}$ is an order-enriched functor and $\phi: F T_{\mathcal{C}} \rTo T_{\mathcal{D}}F $ is a natural transformation such that the diagrams

\begin{diagram}
&  & F \\
& \ldTo^{F\eta ^{\mathcal{C}}}& \dTo^{\eta_F ^{\mathcal{D}}} \\
F T_{\mathcal{C}} & \rTo^{\phi } & T_{\mathcal{D}}F\\
\end{diagram}
and
\begin{diagram}
FT_{\mathcal{C}}T_{\mathcal{C}} & \rTo{ \phi_{T_{\mathcal{C} }}} & T_{\mathcal{D}}FT_{\mathcal{C}} & \rTo^{T_{\mathcal{D}}\phi } & T_{\mathcal{D}}T_{\mathcal{D}}F \\
\dTo^{F\mu ^{\mathcal{C}}} &  &  &  & \dTo^{\mu_F ^{\mathcal{D}}}
\\
FT_{\mathcal{C}} &  & \rTo{\phi } &  & T_{\mathcal{D}}F
\end{diagram}
both commute.

Note that monad opfunctors compose in an obvious manner: given $(F,\phi):(\mathcal{C},\mathbb{T}_{C}) \rTo (\mathcal{D},\mathbb{T}_{D})$ and $(G,\psi):(\mathcal{D},\mathbb{T}_{D}) \rTo (\mathcal{E},\mathbb{T}_{\mathcal{E}})$ the composition $(G,\psi)\circ (F,\phi)$ is given by $(G \circ F, GF T_{\mathcal{C}} \rTo{G\phi} GT_{\mathcal{D}}F \rTo{ \psi_F} T_\mathcal{E}GF)$. Further a transformation between monad opfunctors $(F,\phi)$ and $(F', \phi')$ can be defined as a natural transformation $\alpha: F \rTo F'$ such that $\phi' \alpha_{T_{\mathcal{C}}} = (T_{\mathcal{D}} \alpha )\phi$ so a 2-category whose objects are order-enriched categories with a monad is defined.

The following result gives an alternative description of this 2-category in terms of \emph{liftings} to Kleisli categories: if $F:\mathcal{C}\rTo\mathcal{D}$ is an order-enriched functor then a \emph{lifting} of $F$ is an order-enriched functor $\overline{F}:\mathcal{C}_{\mathbb{T}_{\mathcal{C}}} \rTo   \mathcal{D}_{\mathbb{T}_{\mathcal{D}}}$ such that the square

\begin{diagram}
\mathcal{C}_{\mathbb{T}_{\mathcal{C}}} & \rTo^{\overline{F}} & \mathcal{D}_{\mathbb{T}_{\mathcal{D}}} \\
\uTo^{\mathbb{T}_{\mathcal{C}}} &  & \uTo^{\mathbb{T}_{\mathcal{D}}} \\
\mathcal{C} & \rTo^{F} & \mathcal{D}
\end{diagram}
commutes.
Liftings compose in an obvious manner. The data for a transformation between liftings $(F,\overline{F})$ and $(F', \overline{F'})$ is a pair of natural transformations $\alpha: F \rTo F'$ and $\overline{\alpha}: \overline{F} \rTo \overline{F'}$ such that $\mathbb{T}_{\mathcal{D}}\alpha = \overline{\alpha}_{\mathbb{T}_{\mathcal{C}}}$.

\begin{lemma}
\label{basic adjunction lemma}Given two order-enriched monads $(T_{\mathcal{C}},\eta ^{\mathcal{C}},\mu ^{\mathcal{C}})$ and $(T_{\mathcal{D}},\eta ^{\mathcal{D}},\mu ^{\mathcal{D}})$ there is a bijection between monad opfunctors $(F,\phi)$ from $(\mathcal{C},\mathbb{T}_{C})$ to $ (\mathcal{D},\mathbb{T}_{D})$ and pairs of functors $(F,\overline{F})$ where $\overline{F}$ is a lifting of $F$. This bijection sends a lifting $(F, \overline{F})$ to the monad opfunctor $(F, \phi)$ with $\phi_X =\overline{F}(Id_{T_{\mathcal{C}} X}: T_{\mathcal{C}} X \rTo    X)$, and sends a monad opfunctor $(F,\phi)$ to the lifting $(F, \overline F)$ where $\overline F = F$ on objects and $\overline{F}(f:X \rTo Y) = \phi_Y F(f)$ on morphisms.
Furthermore this bijection preserves composition of monad opfunctors and extends to monad opfunctor transformations; i.e. there is a 2-categorical isomorphism.
\end{lemma}

\begin{proof}
\cite{Ross} (though also see \cite{Pumplun}). The result is a routine application of categorical definitions. The order enrichment aspects are entirely trivial. \cite{Ross} contains the relationship between monad functors and monad algebra morphisms; it also contains a duality which clarifies that Kleisli constructions are dual to monad algebra constructions.
\end{proof}

\begin{corollary}\label{basic_adjunction_corollary}
If $F_1,F_2: \mathcal{C} \rTo \mathcal{D}$ are two order-enriched functors and there is a monad opfunctor $\phi_1 : F_1 T_{\mathcal{C}} \rTo T_{\mathcal{D}} F_1 $ and a natural isomorphism $\rho: F_1 \rTo^{\cong} F_2$ then there is another monad opfunctor $\phi_2 : F_2 T_{\mathcal{C}} \rTo T_{\mathcal{D}} F_2 $ given by $\phi_2 = (T_{\mathcal{D}}\rho)\phi_1 \rho^{-1}_{T_{\mathcal{C}}}$ and a natural isomorphism $\overline{\rho}: \overline{F_1} \rTo^{\cong} \overline{F_2}$ given by $\overline{\rho}_X =F_1(X) \rTo^{\rho_X} F_2 (X) \rTo^{\eta^{\mathcal{D}}_{F_2(X)}} T_{\mathcal{D}} F_2(X)$ for each object $X$ of $\mathcal{C}$
\end{corollary}
\begin{proof}
It is routine to check that $(F_2,\phi_2)$ so defined is a monad opfunctor and that with this definition $\rho :  (F_1,\phi_1) \rTo (F_2,\phi_2)$ is a monad opfunctor transformation. $\overline{\rho}$ is the corresponding lifting (i.e. the image of the 2-cell under the 2-categorical isomorphism of the lemma).
\end{proof}

Next, by double application of the lemma, we show that if a functor between categories (each with a monad) has a left adjoint, then the adjunction lifts to the respective Kleisli categories if and only if there is a monad opfunctor on the right adjoint which is an isomorphism.

\begin{lemma}
\label{adjunction lemma}The following data on an order-enriched adjunction $L\dashv R:%
\mathcal{D}\pile{\rTo \\ \lTo} \mathcal{C}$ and order-enriched monads $(T_{%
\mathcal{D}},\eta ^{\mathcal{D}},\mu ^{\mathcal{D}})$ $\ $and $(T_{\mathcal{C%
}},\eta ^{\mathcal{C}},\mu ^{\mathcal{C}})$ on $\mathcal{D}$ and $\mathcal{C}
$ is equivalent:

(i) a monad isomorphism $\phi :RT_{\mathcal{C}} \rTo T_{\mathcal{D}}R$

(ii) an adjunction $\overline{L}\dashv \overline{%
R}:\mathcal{D}_{\mathbb{T}_{\mathcal{D}}}\pile{\rTo \\ \lTo} \mathcal{C}_{%
\mathbb{T}_{\mathcal{C}}}$ that lifts $L \dashv R$.
\end{lemma}
Note that when, in (ii), it says that an adjunction is a lifting of another adjunction, this is implying that not only are there two liftings $\overline{R}$ and $\overline{L}$ but also that the unit and counit of the lifted adjunction are the liftings of the unit and counit of the original adjunction $L \dashv R$; i.e. the lifting is 2-categorical.
\begin{proof}
Firstly let us say that we have a monad isomorphism $\phi :RT_{\mathcal{C}} \rTo T_{\mathcal{D}}R$. Then certainly there exists $\overline{R}$ by the previous lemma. But also we can construct a monad opfunctor $(L,\phi')$ by defining $\phi'_W: L T_{\mathcal{D}}W \rTo T_{\mathcal{C}} LW$ to be the adjoint transpose (under $L \dashv R$) of $T_{\mathcal{D}}W \rTo^{T_{\mathcal{D}}\eta_W}T_{\mathcal{D}}RLW \rTo^{\phi^{-1}_{LW}}R T_{\mathcal{C}}LW$. By application of the previous lemma we therefore also have $\overline{L}:\mathcal{D}_{\mathbb{T}_{\mathcal{D}}} \rTo \mathcal{C}_{\mathbb{T}_{\mathcal{C}}}$. The unit $\eta : Id \rTo RL$ and counit $ \epsilon : LR \rTo Id$ are, it can be verified, monad opfunctor transformations and so $(L,\phi') \dashv (R,\phi)$ since the triangular identities hold. Therefore since the previous lemma was 2-categorical it effectively shows how to construct two natural transformations $\overline{\eta} : Id \rTo \overline{R}\overline{L}$ and $ \overline{\epsilon} : \overline{L}\overline{R} \rTo Id$ and therefore a lifting $\overline{L}\dashv \overline{R}:\mathcal{D}_{\mathbb{T}_{\mathcal{D}}}\pile{\rTo \\ \lTo} \mathcal{C}_{\mathbb{T}_{\mathcal{C}}}$ as required. \\
In the other direction, assume that we are given such a lifting. Then let $(R,\phi)$ be the monad opfunctor that exists due to the lifting $\overline R$; i.e. $\phi_X:R T_{\mathcal{C}}X \rTo T_{\mathcal{D}} R X = \overline R (Id_{T_{\mathcal{C}}X})$ for any object $X$ of $\mathcal{C}$. We need to construct $\phi^{-1}: T_{\mathcal{D}}R \rTo R T_{\mathcal{C}}$. Consider $\psi : T_{\mathcal{D}}R \rTo R T_{\mathcal{C}}$ defined by $\psi_X$ = the adjoint transpose via $L \dashv R$ of $L T_{\mathcal{D}}R X \rTo^{\phi'_{RX}} T_{\mathcal{C}} LRX \rTo^{T_{\mathcal{C}}\epsilon_X} T_{\mathcal{C}}X$ where $\phi' : L T_{\mathcal{D}} \rTo T_{\mathcal{C}} L$ is from the monad opfunctor $(L,\phi')$ which is derived from the lifting $\overline L$. Then use the triangular identities on $\overline L \dashv \overline R$ to show that $\psi = \phi^{-1}$.

\end{proof}

\begin{corollary}\label{Rights_are_same}
Given two order-enriched adjunctions $L_i\dashv R_i : \mathcal{D}\pile{\rTo \\ \lTo} \mathcal{C}$ $i=1,2$ such that $L_1 = L_2$ and so $R_1 \cong R_2$ via a canonical natural isomorphism $\rho$. If both adjunctions have a monad isomorphism $\phi_i:R_i T_{\mathcal{C}} \rTo T_{\mathcal{D}} R_i$ ($i=1,2$) then provided $\overline{L_1}=\overline{L_2}$, it follows that $\overline{R_1} \cong \overline{R_2}$ via $\overline{\rho}$ where $\overline{\rho}$ is given by $\overline{\rho}_X =R_1(X) \rTo^{\rho_X} R_2 (X) \rTo^{\eta^{\mathcal{D}}_{R_2 (X)}} T_{\mathcal{D}} R_2(X)$.
\end{corollary}
\begin{proof}
By Corollary \ref{basic_adjunction_corollary} it is sufficient to prove that $\phi_2 = (T_{\mathcal{D}}\rho)\phi_1 \rho^{-1}_{T_{\mathcal{C}}}$; we outline a proof that $\phi^{-1}_2 =  \rho_{T_{\mathcal{C}}}\phi^{-1}_1 (T_{\mathcal{D}}\rho^{-1})$ from which this follows. Since $\overline{L_1}=\overline{L_2}$ we have that the corresponding monad opfunctors are the same by Lemma \ref{basic adjunction lemma}, i.e. $\phi'_1 : L_1 T_{\mathcal{D}} \rTo T_{\mathcal{C}}L_1$ is equal to $\phi'_2 : L_2 T_{\mathcal{D}} \rTo T_{\mathcal{C}}L_2$. However the last lemma has established that $\phi^{-1}_i$ is uniquely determined by $\phi'_i$ as it is the adjoint transpose via $L_i \dashv R_i$ of $(T_{\mathcal{C}}\epsilon_i)(\phi'_i)_R$ for $i=1,2$. It just therefore remains to check that

\begin{eqnarray*}
T_{\mathcal{D}} R_2 \rTo^{(\eta_2)_{T_{\mathcal{D}}R_2} } R_2 L_2 T_{\mathcal{D}}R_2 \rTo^{R_2 (\phi'_2)_{R_2}} R_2 T_{\mathcal{C}} L_2 R_2 \rTo^{R_2 T_{\mathcal{C}}(\epsilon_2)} R_2 T_{\mathcal{C}}\\
\end{eqnarray*}
factors as $\rho_{T_{\mathcal{C}}}\phi^{-1}_1 (T_{\mathcal{D}}\rho^{-1})$ which follows by naturality since $\phi^{-1}_1$ factors as $(R_1 T_{\mathcal{C}}(\epsilon_1)) R_1(\phi'_1)_{R_1} (\eta_1)_{T_{\mathcal{D}}R_1}$ and $\phi'_1=\phi'_2$.
\end{proof}

\section{Preserving the strength}

In this section we prove a proposition (Proposition \ref{frobenius implies commutes with PP} below) which will be a key step in proving that the categories (i) and (ii) of the main theorem (Theorem \ref{maintheorm}) are equivalent. The proof of this proposition will be by application of the last lemma (Lemma \ref{adjunction lemma}). The proposition shows that the pullback adjunction $\Sigma_f \dashv f^*:
\mathbf{Loc}_{\mathcal{F}}\pile{\rTo \\ \lTo} \mathbf{Loc}_{\mathcal{E}}$ which arises from a geometric morphism $f: \mathcal{F} \rTo \mathcal{E}$ satisfies the conditions of (ii) of the main theorem: $f^*$ preserves finitary coproduct, there exists a monad isomorphism $\phi: f^* \mathbb{P}_{\mathcal{E}} \rTo^{\cong} \mathbb{P}_{\mathcal{F}}f^*$ and $f^*$ preserves the strength on $\mathbb{P}$ up to the monad isomorphism $\phi$.   In fact this is all known except for the assertion that the strength is preserved. That $f^{*}$ preserves finitary coproduct is well known (for example, covered in the proof of Proposition 24 of \cite{towdcpo}) and the fact that $f^*$ preserves $\mathbb{P}$ is known (e.g. \cite{DoubPt}); that, further, the whole monad structure is preserved follows from Lemma \ref{adjunction lemma} since Proposition 24 of \cite{towdcpo} exhibits a lifting of $\Sigma_f \dashv f^*$ to the Kleisli categories.

Our proof on the preservation of the strength will make use of the result, \cite{towgeom}, that $\Sigma_f \dashv f^*$ satisfies Frobenius reciprocity. To proceed we must now set out clearly all the required definitions:
\begin{definition}
An adjunction $L\dashv R:\mathcal{D}\pile{\rTo \\ \lTo} \mathcal{C}$ between
cartesian categories satisfies \emph{Frobenius reciprocity} provided the map
$L(R(X)\times W)\rTo^{(L\pi _{1},L\pi _{2})}LRX\times LW\rTo^{\varepsilon _{X}\times Id_{LW}}X\times LW$ is an isomorphism for all objects $W$ and $X$ of $\mathcal{D}$
and $\mathcal{C}$ respectively.
\end{definition}
\begin{definition}
The \emph{double power monad} on $\mathbf{Loc}$ is given by double exponentiation at the Sierpi\'{n}ski locale $\mathbb{S}$. The functor part is $X \mapsto \mathbb{S}^{\mathbb{S}^X}$. The unit and multiplication parts of the monad are immediate from the definition of the functor part as an exponential.
\end{definition}
It is important to be aware that ${\mathbb{S}^X}$ is the presheaf $\mathbf{Loc}(\_ \times X , \mathbb{S})$ and that the exponentiation takes place in $[\mathbf{Loc}^{op}, \mathbf{Set}]$. See \cite{victow} for the relevant background. Consequently the Kleisli category of the double power locale monad is equivalent to the opposite of the full subcategory of $[\mathbf{Loc}^{op}, \mathbf{Set}]$ consisting only of objects of the form $\mathbb{S}^X$. We use $\boxtimes_X:\mathbb{S}^X \rTo \mathbb{S}^{\mathbb{P}X}$ as notation for the universal natural transformation (i.e. the double exponential transpose of the map $\mathbb{P}X \rTo^{\cong} \mathbb{S}^{\mathbb{S}^X}$). This natural transformation is universal in the sense that every natural transformation $\alpha: \mathbb{S}^Y \rTo \mathbb{S}^X$ factors as $\mathbb{S}^{f_{\alpha}} \boxtimes_Y:\mathbb{S}^Y \rTo \mathbb{S}^X$ for some unique Kleisli map $f_{\alpha}: X \rTo \mathbb{P}Y$. It is worth clarifying, since this exponential relative to $[\mathbf{Loc}^{op}, \mathbf{Set}]$ takes a central role in what follows, that if $\alpha: \mathbb{S}^Y \rTo \mathbb{S}^X$ is a natural transformation then for any locale $Z$ the exponential map $\alpha^Z: \mathbb{S}^{Z \times Y} \rTo \mathbb{S}^{ Z \times X }$ exists; it is given by $\alpha^Z_{X'}=\alpha_{X' \times Z }$ at any locale $X'$.
In the case that there exists $\phi$ a monad opfunctor between double power locale monads (for a given functor $R: \mathbf{Loc}_{\mathcal{E}} \rTo \mathbf{Loc}_{\mathcal{F}}$) we have that the lifting preserves the universal natural transformation up to $\phi$:
\begin{lemma}
\label{phi_as_ext}
Given an order-enriched functor $R: \mathbf{Loc}_{\mathcal{E}} \rTo  \mathbf{Loc}_{\mathcal{F}}$ and a monad opfunctor $\phi: R \mathbb{P}_{\mathcal{E}} \rTo \mathbb{P}_{\mathcal{F}} R$, the following diagram commutes:
\begin{diagram}
&                                                                 & \mathbb{S}_{\mathcal{F}}^{ \mathbb{P}_{\mathcal{F}} R X  }\\
 & \ruTo^{\boxtimes_{R X}} & \dTo_{\mathbb{S}_{\mathcal{F}}^{ \phi_X}} \\
\mathbb{S}_{\mathcal{F}}^{R X} & \rTo^{\overline{R}^{op}\boxtimes_X} & \mathbb{S}_{\mathcal{F}}^{R\mathbb{P}_{\mathcal{E}} X}   \\
 \end{diagram}
where $\overline{R}$ is the lifting corresponding to $\phi$.
\end{lemma}
\begin{proof}
Recall (Lemma \ref{basic adjunction lemma}) that for a Kleisli morphism $f: X \rTo \mathbb{P}_{\mathcal{E}}Y$, $\overline{R}f$ is defined to be $\phi_Y Rf$. The result then follows since the Kleisli morphism corresponding to $\boxtimes_X$ is the identity on $\mathbb{P}_{\mathcal{E}}X$.
\end{proof}
\begin{definition}
(1). The \emph{strength} of the double power monad is given by the natural transformation $t:\mathbb{P}(\_)\times (\_) \rTo \mathbb{P}(\_ \times \_)$ defined by
\begin{eqnarray*}
t_{X,Y}: \mathbb{P} X \times Y \rTo \mathbb{P}(X \times Y)\
\end{eqnarray*}
for any locales $X$ and $Y$, where $t_{X,Y}$ is the Kleisli morphism corresponding to the exponential $\boxtimes_X^Y:\mathbb{S}^{X \times Y} \rTo \mathbb{S}^{\mathbb{P}X \times Y}$ in $[\mathbf{Loc}^{op},\mathbf{Set}]$.

(2). For any order-enriched functor $R: \mathbf{Loc}_{\mathcal{F}} \rTo \mathbf{Loc}_{\mathcal{E}}$ with a monad opfunctor $\phi: R \mathbb{P}_{\mathcal{E}} \rTo \mathbb{P}_{\mathcal{F}} R$, \emph{$R$ preserves the strength up to $\phi$} if and only if the diagram
\begin{diagram}
R ( \mathbb{P}_{\mathcal{E}} X \times Y ) &   &  \rTo^{R t^{\mathcal{E}}_{X,Y}} &  & R\mathbb{P}_{\mathcal{E}}(X \times Y)   \\
\dTo^{(R \pi_1, R\pi_2)} &  & & &  \dTo_{\phi_{X\times Y}}\\
R\mathbb{P}_{\mathcal{E}}X \times R Y& & & & \mathbb{P}_{\mathcal{F}}(R(X \times Y))  \\
\dTo^{\phi_X \times Id_{RY}}  & & & &  \dTo_{\mathbb{P}_{\mathcal{F}}(R \pi_1, R\pi_2)}\\
\mathbb{P}_{\mathcal{F}}RX \times RY &  &  \rTo^{t^{\mathcal{F}}_{RX,RY}}  &   & \mathbb{P}_{\mathcal{F}}(RX \times RY)  \\
\end{diagram}
commutes for all locales $X$ and $Y$ over $\mathcal{E}$.
\end{definition}

The next lemma provides a characterization of when the strength is preserved in terms of the lifting $\overline{R}$ induced by $\phi$:

\begin{lemma}
\label{characterizepreservingstrength}Given a monad opfunctor $\phi: R \mathbb{P}_{\mathcal{E}} \rTo \mathbb{P}_{\mathcal{F}} R$ then the following are equivalent:

1. $R$ preserves the strength up to $\phi$,

2. for any locales $X$ and $Y$ over $\mathcal{E}$ the diagram

\begin{diagram}
 \mathbb{S}_{\mathcal{F}}^{R X \times R Y}  &  & \rTo^{ \mathbb{S}_{\mathcal{F}}^{ (R \pi_1, R \pi_2)}}  &  &    \mathbb{S}_{\mathcal{F}}^{R (X \times  Y)} \\
\dTo^{\boxtimes_{R X}^{RY}} & & &         &  \dTo_{\overline{R}^{op}\boxtimes_X^Y} \\
 \mathbb{S}_{\mathcal{F}}^{ \mathbb{P}_{\mathcal{F}}R X \times R Y} & \rTo^{\mathbb{S}_{\mathcal{F}}^{ \phi_X \times Id}} &    \mathbb{S}_{\mathcal{F}}^{R\mathbb{P}_{\mathcal{E}} X \times R Y}   & \rTo^{ \mathbb{S}_{\mathcal{F}}^{ (R \pi_1, R \pi_2)}}  & \mathbb{S}_{\mathcal{F}}^{R(\mathbb{P}_{\mathcal{E}} X \times  Y)}\\
\end{diagram}
commutes where $\overline{R}$ is the lifting corresponding to $\phi$; and,

3. for any natural transformation $\alpha: \mathbb{S}_{\mathcal{E}}^{X_1} \rTo \mathbb{S}_{\mathcal{E}}^{X_2}$ and for any locale $Y$ over $\mathcal{E}$, $ \overline{R}^{op}(\alpha^Y) \mathbb{S}_{\mathcal{F}}^{ (R \pi_1, R \pi_2)}= \mathbb{S}_{\mathcal{F}}^{ (R \pi_1, R \pi_2)}(\overline{R}^{op}\alpha)^{RY} $.

If further $R$ preserves binary product then the above is also equivalent to:

4. for any natural transformation $\alpha: \mathbb{S}_{\mathcal{E}}^{X_1} \rTo \mathbb{S}_{\mathcal{E}}^{X_2}$ and for any locale $Y$ over $\mathcal{E}$, $\overline{R}^{op}(\alpha^Y) \cong (\overline{R}^{op}\alpha)^{RY}$ via the natural isomorphism $\mathbb{S}_{\mathcal{F}}^{ (R \pi_1, R \pi_2)}$.
\end{lemma}

\begin{proof}
The equivalence of 1. and 2. is a question of taking the exponential transpose of the diagram that determines whether $R$ preserves the strength. For the equivalence of 2. and 3. recall that each natural transformation $\alpha: \mathbb{S}^Y \rTo \mathbb{S}^X$ factors as $\mathbb{S}^{f_{\alpha}} \boxtimes_Y:\mathbb{S}^Y \rTo \mathbb{S}^X$ for some unique Kleisli map $f_{\alpha}: X \rTo \mathbb{P}Y$. The equivalence of 3. and 4. is immediate.  \end{proof}
\begin{proposition}
\label{frobenius implies commutes with PP}
For any geometric morphism $f: \mathcal{F} \rTo \mathcal{E}$ the resulting pullback adjunction $\Sigma_f \dashv f^*:
\mathbf{Loc}_{\mathcal{F}}\pile{\rTo \\ \lTo} \mathbf{Loc}_{\mathcal{E}}$ has the properties (a) $f^*$ preserves finitary coproduct, (b) there exists a monad isomorphism $\phi: f^* \mathbb{P}_{\mathcal{E}} \rTo^{\cong} \mathbb{P}_{\mathcal{F}}f^*$ and (c) $f^*$ preserves the strength on $\mathbb{P}$ up to the monad isomorphism $\phi$.
\end{proposition}
\begin{proof}
As has been covered already in the first paragraph of this Section, all that is required is a proof of (c). To prove (c) we need to be explicit about how $\phi$ is constructed. It is well known that there is an isomorphism $f^*  \mathbb{S}_{\mathcal{E}}\cong \mathbb{S}_{\mathcal{F}}$ (i.e. the Sierpi\'{n}ski locale is preserved by pullback along a geometric morphism) and, see \cite{towgeom}, it is known that $\Sigma_f \dashv f^*$ satisfies Frobenius reciprocity. From this we can apply Proposition 5.1 of \cite{towgeom} and construct a lifting $\overline{f^*} \dashv \overline{\Sigma_f}: (\mathbf{Loc}_{\mathcal{F}})_{\mathbb{P}_{\mathcal{F}}}\pile{\rTo \\ \lTo} (\mathbf{Loc}_{\mathcal{E}})_{\mathbb{P}_{\mathcal{E}}}$ to the Kleisli categories and so this gives rise to a monad isomorphism $\phi$. For example if $\alpha :\mathbb{S}_{\mathcal{E}}^{X}\rTo \mathbb{S}_{\mathcal{E}}^{X^{\prime }}$ is a natural transformation, define $\overline{f^*}^{op}(\alpha )$ by
\begin{diagram}
\mathbf{Loc}_{\mathcal{F}}(W\times f^*X,\mathbb{S}_{\mathcal{F}}) & \rTo^{[\overline{f^*}^{op}(\alpha )]_{W}} & \mathbf{Loc}_{\mathcal{F}}(W\times f^*X^{\prime },\mathbb{S}_{\mathcal{F}}) \\
\dTo_{\cong} &  & \uTo^{\cong} \\
\mathbf{Loc}_{\mathcal{F}}(W\times f^*X,f^*\mathbb{S}_{\mathcal{E}}) &  & \mathbf{Loc}_{\mathcal{F}}(W\times f^*X^{\prime },f^* \mathbb{S}_{\mathcal{E}}) \\
\dTo_{\cong} &  & \uTo^{\cong} \\
\mathbf{Loc}_{\mathcal{E}}(\Sigma_f W\times X,\mathbb{S}_{\mathcal{E}}) & \rTo{\alpha _{\Sigma_f W}} & \mathbf{Loc}_{\mathcal{E}}(\Sigma_f W\times X^{\prime}, \mathbb{S}_{\mathcal{E}})\\
\end{diagram}
for any locale $W$ of $\mathcal{F}$, where the vertical morphisms are isomorphisms by our observation that Frobenius reciprocity is satisfied and $f^*$ preserves the Sierpi\'{n}ski locale.
To prove that $f^*$ preserves the strength we show that for any natural transformation $\alpha: \mathbb{S}_{\mathcal{E}}^{X_1} \rTo \mathbb{S}_{\mathcal{E}}^{{X_2}}$, $\overline{f^*}^{op}(\alpha^Y)  \cong (\overline{f^*}^{op}\alpha)^{RY}$ and appeal to the previous lemma. Now, for any locale $W$ over $\mathcal{F}$ we have that $[\overline{f^*}^{op}(\alpha^Y)]_W$ is the composite
\begin{eqnarray*}
{\mathbf{Loc}}_{\mathcal{F}}(f^*(X_1 \times Y) \times W, \mathbb{S}_{\mathcal{F}}) & \rTo^{\cong} &  \\
{\mathbf{Loc}}_{\mathcal{E}}(X_1 \times Y \times \Sigma_f W, \mathbb{S}_{\mathcal{E}})
& \rTo^{\alpha_{Y \times \Sigma_f W}} & {\mathbf{Loc}}_{\mathcal{E}}(X_2 \times Y \times \Sigma_f W, \mathbb{S}_{\mathcal{E}}) \\
& \rTo^{\cong} & {\mathbf{Loc}}_{\mathcal{F}}(f^*(X_2 \times Y) \times W, \mathbb{S}_{\mathcal{F}})
\end{eqnarray*}
and $(\overline{f^*}^{op}\alpha)_W^{f^*Y}$ is the composite,
\begin{eqnarray*}
{\mathbf{Loc}}_{\mathcal{F}}(f^*X_1 \times f^*Y \times W, \mathbb{S}_{\mathcal{F}}) & \rTo^{\cong}  \\ {\mathbf{Loc}}_{\mathcal{E}}(X_1 \times \Sigma_f( f^* Y \times W), \mathbb{S}_{\mathcal{E}})
& \rTo^{\alpha_{\Sigma_f(f^*Y \times W)}} & {\mathbf{Loc}}_{\mathcal{E}}(X_2 \times \Sigma_f ( f^* Y \times W), \mathbb{S}_{\mathcal{E}}) \\  { } &  \rTo^{\cong} & {\mathbf{Loc}}_{\mathcal{F}}(f^*X_2 \times f^*Y \times W, \mathbb{S}_{\mathcal{F}})
\end{eqnarray*}
and so $\overline{f^*}^{op}(\alpha^Y) \cong (\overline{f^*}^{op}\alpha)^{f^*Y}$ via the natural isomorphism $\mathbb{S}_{\mathcal{F}}^{ (f^* \pi_1, f^* \pi_2)}$ by application of naturality of $\alpha$ at the isomorphism $\Sigma_f(f^*Y \times W) \cong Y \times \Sigma_f W$.
\end{proof}
\section{Representing an adjunction between categories of locales using a geometric morphism}
In this section we prove that if we are given an order-enriched adjunction $L\dashv R:
\mathbf{Loc}_{\mathcal{F}}\pile{\rTo \\ \lTo} \mathbf{Loc}_{\mathcal{E}}$ and a monad isomorphism $\phi:R\mathbb{P}_{\mathcal{E}}\cong \mathbb{P}_{\mathcal{F}}R$ such that (a) $R$ preserves finitary coproduct and (b) $R$ preserves, up to the monad isomorphism $\phi$, the strength on $\mathbb{P}$, then we can construct a geometric morphism $f: \mathcal{F} \rTo \mathcal{E}$, unique up to natural isomorphism, such that $R \cong f^*$ and $L \cong \Sigma_f$. This construction (Proposition \ref{represent_geom_morphism}) combined with the previous proposition (Proposition \ref{frobenius implies commutes with PP}) proves the equivalence of the categories (i) and (ii) of the main theorem (Theorem \ref{maintheorm}).

The next three lemmas set out various lattice theoretic results needed in the proof of Proposition \ref{represent_geom_morphism}. We start by setting notation related to dcpo homomorphisms between frames. If $q:\Omega X_1 \rTo \Omega X_2$ is a dcpo homomorphism then it corresponds to a natural transformation $\mathbb{S}^{X_1} \rTo \mathbb{S}^{X_2}$ (\cite{victow}). If $Y$ is some other locale then we use $q^Y$ to denote the dcpo homomorphism $\Omega(Y \times X_1 ) \rTo \Omega (Y \times X_2 )$ that is equal to the natural transformation corresponding to $q$ evaluated at $Y$. This notation is consistent with our description of the exponential $\alpha^Y$ in $[\mathbf{Loc}^{op}, \mathbf{Set}]$ since the dcpo homomorphism corresponding to $\alpha^Y$ is $q^Y$ if $q$ is the dcpo homomorphism corresponding to $\alpha$. Note that if there is $R:\mathbf{Loc}_{\mathcal{E}} \rTo \mathbf{Loc}_{\mathcal{F}}$ and a monad morphism $\phi:R\mathbb{P}_{\mathcal{E}}\rTo  \mathbb{P}_{\mathcal{F}}R$, then we will also use the notation $\overline{R}^{op}(q)$ to denote the effect of $R$ on the dcpo homomorphisms corresponding to natural transformations. In other words no distinction is going to be made between $\overline{R}^{op}$ acting on natural transformations and $\overline{R}^{op}$ acting on dcpo homomorphisms. Note that by Lemma \ref{characterizepreservingstrength} if further $R$ preserves binary product and the strength (up to $\phi$) then $\overline{R}^{op}(q^Y) \cong (\overline{R}^{op}q)^{RY}$ via the canonical isomorphism $\Omega_{\mathcal{F}}(R(\_) \times R(\_)) \cong \Omega_{\mathcal{F}}(R(\_ \times \_))$

The following lemma allows us to capture maps to frames as dcpo homomorphisms. Since the property of commuting with the double power locale monad will only give us information about the preservation of dcpo homomorphisms this is a key technical step.
\begin{lemma}\label{omegaXtoA}
For any set $A$ and locale $X$ over a topos $\mathcal{E}$ there is an order isomorphism
\begin{eqnarray*}
\mathcal{E}(A,\Omega X)\cong\mathbf{dcpo}(\Omega 0,\Omega( A \times X))
\end{eqnarray*}
natural in dcpo homomorphisms between $\Omega X$ and functions between $A$.
\end{lemma}
If we follow a notation that $\psi' : \Omega 0 \rTo \Omega (A \times X)$ is the mate of $\psi : A \rTo \Omega X$ under this order isomorphism then the naturality assertion with respect to dcpo homomorphisms is that for any $q: \Omega X_1 \rTo \Omega X_2$ then $(q \psi)'= q^A \psi'$. The assertion of naturality with respect to functions is that if $f: B \rTo A$ is a map then $( \psi f)'= [\Omega(f)]^X \psi'$. Note that in this case $[\Omega(f)]^X$ is the suplattice tensor $\Omega(f) \otimes Id_{\Omega(X)}$; this can be seen by the construction of the natural transformation representation of any suplattice homomorphism (Theorem 23 in \cite{victow}).
\begin{proof}
This result, without the naturality statement, is well known lattice theory (since, of course, $\Omega 0=1$). The point of $\Omega (A \times X)$ corresponding to $\psi: A \rTo \Omega X$ is given by $\bigvee_{a \in A} \{a \} \otimes \psi a$. For the naturality assertions consult Lemmas 51 and 54 of \cite{towdcpo}.
\end{proof}
We use the notation $\overline{\psi}: PA \rTo \Omega X$ for the free suplattice homomorphism on $\psi: A \rTo \Omega X$. Below we will need an explicit description of $\overline{\psi}$ in terms of $\overline{\psi'}$ to prove the main proposition of this section. This description is provided by the next lemma:
\begin{lemma}
\label{psi_bar_explicit}
Given $\psi: A \rTo \Omega X$, $\overline{\psi}: PA \rTo \Omega X$ factors as
\begin{eqnarray*}
PA \rTo^{\cong} \Omega ( A \times 1 )  \rTo^{{\overline{\psi'}}^A} \Omega( A \times A \times  X )\rTo^{\Omega(\Delta_A \times Id_X)} \Omega (A \times X)    \rTo^{\exists_{\pi_2}} \Omega X
\end{eqnarray*}
\end{lemma}
\begin{proof}
Note first that ${\overline{\psi'}}^A$ is $\overline{\psi'} \otimes Id_{PA}$. Since $PA$ is the free suplattice on $A$ we just need to check that $\psi(a)$ is equal to $  \exists_{\pi_2}\Omega (\Delta_A \times Id_X) ( \{a \} \otimes (\bigvee_{a' \in A} \{a' \} \otimes \psi a' ) ) =  \exists_{\pi_2} (  \{a \} \otimes \psi a ) )$. This is immediate since for any discrete locale $B$ and any other locale $Y$ the map $\exists_{\pi_1}: PB  \otimes \Omega Y \rTo \Omega Y$ is given by $ \exists_{\pi_1}(I \otimes c) = \bigvee_{\exists * \in I} c$.
\end{proof}
Below we will need to specialize the order isomorphism of Lemma \ref{omegaXtoA} to the case that $A$ is a poset. The specialization is provided by the next lemma the proof of which hinges on the explicit description of $\overline{\psi}$ given in the previous lemma.
\begin{lemma}
\label{monotone}
For any poset $A$ and locale $X$ there is an order isomorphism between order preserving (i.e. monotone) maps $A \rTo \Omega X$ and
\begin{eqnarray*}
\{ \psi': \Omega 0 \rTo \Omega (A \times X) | (\uparrow_A)^X \psi' = \psi' \}
\end{eqnarray*}
where $\uparrow_A: PA \rTo PA$ is the upper closure operator with respect to the partial order on $A$.
\end{lemma}
Note that $\uparrow_A$ is a suplattice homomorphism, so it is a dcpo homomorphism and so $(\uparrow_A)^X$ is well defined; as above, it is the familiar suplattice tensor $\uparrow_A \otimes Id_{\Omega X}$.
\begin{proof}
We must check that $\psi : A \rTo \Omega X$ is a monotone map if and only if $(\uparrow^A)^X \psi' = \psi'$, i.e. if and only if
\begin{eqnarray*}
\bigvee_{b \geq a} \{b\} \otimes \psi a =  \bigvee_{a' \in A} \{a'\} \otimes \psi a'   \text{      ($\star$).}
\end{eqnarray*}
Certainly if $\psi$ is monotone then ($\star$) holds since $\{b \} \otimes \psi a \leq \{b \} \otimes \psi b$ for any $ b \geq a$. Conversely if ($\star$) holds then by applying the previous lemma we have that $  \exists_{\pi_2}\Omega (\Delta_A \times Id_X) ( \{a' \} \otimes (\bigvee_{b \geq a} \{b \} \otimes \psi a ) ) = \psi (a')$ for each $a' \in A$. For any $b \geq a$ take $a' = b$ in this last, and note that the LHS is greater than or equal to $\psi a$.
\end{proof}
We now prove a series of lemmas about an order-enriched monad opfunctor $(R,\phi)$ with $R :\mathbf{Loc}_{\mathcal{E}} \rTo  \mathbf{Loc}_{\mathcal{F}}$ and $\phi:R\mathbb{P}_{\mathcal{E}}\rTo  \mathbb{P}_{\mathcal{F}}R$. These lemmas provide more information on what sort of structures are preserved by $R$ relative to the ambient toposes $\mathcal{E}$ and $\mathcal{F}$ given increasing assumptions about $R$ and $\phi$.
\begin{lemma}
\label{discretepreservationlemma}
Given an order-enriched functor $R:
\mathbf{Loc}_{\mathcal{E}} \rTo  \mathbf{Loc}_{\mathcal{F}}$ which preserves binary products and a monad opfunctor $\phi:R\mathbb{P}_{\mathcal{E}}\rTo  \mathbb{P}_{\mathcal{F}}R$, $R$ preserves discrete locales and so defines a functor $\mathcal{E} \rTo \mathcal{F}$.
\end{lemma}
\begin{proof}
Consult Proposition 5.2 of \cite{towgeom}, since in this case $\overline{R}^{op}$ exists and is order-enriched. The proof is done by checking that $R$ preserves open maps and it is shown that if $\exists_f$ is left adjoint to $\Omega_{\mathcal{E}}f$ witnessing that $f : Y \rTo X$ is open, then $\overline{R}^{op}(\exists_f)$ witnesses that $\Omega_{\mathcal{F}}Rf$ is open.
\end{proof}

If we further know that $R$ preserves all finitary limits (as it will do if it is a right adjoint) then it further must define a functor from the category of internal posets in $\mathcal{E}$ to the category of internal posets in $\mathcal{F}$ (i.e. $\mathbf{Pos}_{\mathcal{E}} \rTo \mathbf{Pos}_{\mathcal{F}}$). This is because the property of being an internal poset can be expressed using finitary limits, see e.g. Lemma 14 of \cite{towdcpo}. The next two lemmas provide more consequences of the further assumption that $R$ is cartesian. The first lemma verifies something that we would expect. For any open $a \rInto X$ there is a map $p_a:\Omega_{\mathcal{E}}0\rTo\Omega_{\mathcal{E}}X$ which is the point corresponding to $a$. If $R$ preserves the Sierpi\'{n}ski object then $Ra$ is an open of $RX$ since a subobject $a \rInto X$ (i.e. a regular monomorphism) in $\mathbf{Loc}_{\mathcal{E}}$ is open if and only if there is a pullback diagram
\begin{diagram}
a & \rTo & 1_{\mathcal{E}} \\
\dTo & & \dTo^{1} \\
X & \rTo^{\chi_a} & \mathbb{S}_{\mathcal{E}}
\end{diagram}
for some unique $\chi_a$, where $1:1_{\mathcal{E}}\rTo\mathbb{S}_{\mathcal{E}}$ is the top elements of $\mathbb{S}_{\mathcal{E}}$. Note that $p_a$ factors as $\Omega_{\mathcal{E}}\chi_a \boxtimes_{0_{\mathcal{E}}}$ since $\mathbb{S}_{\mathcal{E}}\cong\mathbb{P}_{\mathcal{E}}0_{\mathcal{E}}$. We would expect that $p_{Ra}$ is equal to $\overline{R}^{op}(p_a)$ if $R$ preserves $0_{\mathcal{E}}$:

\begin{lemma}
\label{pointpreserved}
If $R$ is as in the previous lemma and further is cartesian and preserves the initial locale then for any open $a \rInto X$ of $X$, some locale over $\mathcal{E}$, $p_{Ra}$ is equal to $\overline{R}^{op}(p_a)$ via the canonical isomorphism $\Omega_{\mathcal{F}}0_{\mathcal{F}} \cong \Omega_{\mathcal{E}}(R0_{\mathcal{E}})$.
\end{lemma}
\begin{proof}
$R$ preserves the Sierpi\'{n}ski object since $\mathbb{S}_{\mathcal{E}}\cong\mathbb{P}_{\mathcal{E}}0_{\mathcal{E}}$ and $R$ preserves $0_{\mathcal{E}}$. But $\overline{R}^{op}(\boxtimes_{0_{\mathcal{E}}}) = \Omega_{\mathcal{F}}\phi_X \boxtimes_{0_{\mathcal{F}}}$ (Lemma \ref{phi_as_ext}) and so the proof can be completed by naturality of $\phi$ since $R$ preserves the pullback square that defines $a \rInto X$.
\end{proof}

\begin{lemma}
\label{closure_preserved}
Given the conditions of the previous lemma together with the further assumption that $R$ preserves binary coproduct, then for $(A, \leq_A)$ an internal poset in $\mathcal{E}$ we have that $\overline{R}^{op}(\uparrow_A)= \uparrow_{RA}$.
\end{lemma}
Since $\uparrow_A$ is a suplattice endomorphism on the power set of $A$ certainly $\overline{R}^{op}(\uparrow_A)$ exists since $R$ extends to dcpo homomorphisms.
\begin{proof}
$\uparrow_A$ is equal to the composite
\begin{eqnarray*}
PA \rTo^{(\Omega_{\mathcal{E}}\pi_1,p_{\leq_A}\Omega_{\mathcal{E}}0_A)} (PA \otimes PA )\times (PA \otimes PA) & \rTo^{\wedge_{PA\otimes PA}}& \\
PA \otimes PA & \rTo^{\exists_{\pi_2}} & PA
\end{eqnarray*}
where $0_A:0 \rTo A$ is the unique map from the initial locale and $p_{\leq_A}:\Omega_{\mathcal{E}}0\rTo PA \otimes PA$ is the point of $PA \otimes PA$ corresponding to the partial order on $A$. Since each component in this composite is preserved by $\overline{R}^{op}$ we have that $\uparrow$ is preserved as required. To see that each component is preserved recall that locale coproduct is given by the set theoretic product of the underlying frames. Further the second arrow is meet on $PA \otimes PA $ which is right adjoint to the diagonal and so is preserved as the order enrichment is preserved. That $p_{\leq_{RA}}$ is equal to $\overline{R}^{op}(p_{\leq_A})$ (up to canonical isomorphism) is covered by the previous lemma.
\end{proof}

If we now assume that there is an adjunction $L \dashv R$ with the property that $R$ preserves finitary coproduct, then the extension to dcpo homomorphisms specializes to suplattice homomorphisms:
\begin{lemma}
\label{suplatticespeciallemma}
Given an order-enriched adjunction  $L\dashv R:
\mathbf{Loc}_{\mathcal{F}}\pile{\rTo \\ \lTo} \mathbf{Loc}_{\mathcal{E}}$ with a monad isomorphism $\phi:R\mathbb{P}_{\mathcal{E}}\rTo^{\cong}  \mathbb{P}_{\mathcal{F}}R$ such that $R$ preserves finitary coproduct then, on morphisms, both $\overline{L}^{op}$ and $\overline{R}^{op}$ preserve the property of being a suplattice homomorphism and so for any locales $X$ and $W$ over $\mathcal{E}$ and $\mathcal{F}$ respectively, there is a natural order isomorphism
\begin{eqnarray*}
\mathbf{Sup}_{\mathcal{E}}(\Omega_{\mathcal{E}}X, \Omega_{\mathcal{E}}LW) \rTo^{\cong} \mathbf{Sup}_{\mathcal{F}}(\Omega_{\mathcal{F}}RX, \Omega_{\mathcal{F}}W)
\end{eqnarray*}
given by $\Omega_{\mathcal{F}} \eta_{W} \overline{R}^{op}( \_ )$.
\end{lemma}
\begin{proof}
The second assertion is immediate from the first. Recall that finitary locale coproduct is given by set theoretic product. Localic codiagonal is set theoretic diagonal. If finitary set theoretic diagonal is preserved by an order-enriched functor then so is its left adjoint, which is finitary join. So, $\overline{L}^{op}$ and $\overline{R}^{op}$ both preserve join preserving maps and so specialize to suplattice homomorphisms provided $L$ and $R$ preserve finitary coproduct. $R$ preserves finitary coproduct by assumption and $L$ preserves finitary coproduct since it is a left adjoint.
\end{proof}
Our final basic lattice theoretic result is the following lemma on presenting any set as a dcpo equalizer.
\begin{lemma}\label{dcpoequalizer}
For any object $B$ of a topos $\mathcal{F}$ there is an equalizer diagram
\begin{equation*}
 B\rInto PB \pile{\rTo^{q_B} \\ \rTo_{q'_B}}  P(B\times B)\times \Omega_{\mathcal{F}}
\end{equation*}
in $\mathcal{F}$ where $q_{B}$ and $q'_{B}$ are dcpo homomorphisms.
\end{lemma}
\begin{proof}
$P(B\times B)\times \Omega_{\mathcal{F}} $ is the frame of the locale $(B\times B)+1$. Recall $\Omega_{\mathcal{F}} =P1$, and $1=\{\ast \}$ the singleton set.
Let $q_{B}(I)=(I\times I,\{\ast \})$ and let $q'_{B}(I)=(\{(i,i)\mid i\in I\},\{\ast \mid \exists i\in I\})$. It is routine to verify that
these are both dcpo homomorphisms and that $B$ is their equalizer.
\end{proof}
\begin{proposition}\label{represent_geom_morphism}
If $L\dashv R:
\mathbf{Loc}_{\mathcal{F}}\pile{\rTo \\ \lTo} \mathbf{Loc}_{\mathcal{E}}$ is an order-enriched adjunction with a monad isomorphism $\phi:R\mathbb{P}_{\mathcal{E}}\cong \mathbb{P}_{\mathcal{F}}R$ such that (a) $R$ preserves finitary coproduct and (b) $R$ preserves, up to the monad isomorphism $\phi$, the strength on $\mathbb{P}$ then there exists a geometric morphism $f:\mathcal{F}\rTo\mathcal{E}$, unique up to natural transformation, such that $R\cong f^*$ and $L \cong \Sigma_f$.
\end{proposition}
\begin{proof}
From Proposition \ref{discretepreservationlemma} $R$ defines a functor $f^*:\mathcal{E}\rTo\mathcal{F}$ by restriction to discrete locales. It is our candidate for the inverse image of a geometric morphism. Note that since finite limits of discrete locales are created in $\mathbf{Loc}$ we have that $f^*$ is cartesian. To construct a right adjoint to $f^*:\mathcal{E}\rTo\mathcal{F}$ we use Lemma \ref{dcpoequalizer}. For any object $B$ of $\mathcal{F}$ define $f_*(B)$ to be the equalizer in $\mathcal{E}$ of

\begin{equation*}
\overline{L}^{op}(PB)
\begin{array}{c}
\overset{\overline{L}^{op}q_{B}}{\rTo } \\
\underset{\overline{L}^{op}q'_{B}}{\rTo }
\end{array}
\overline{L}^{op}(P(B\times B)\times \Omega_{\mathcal{F}})
\end{equation*}
where $q_{B}$ and $q'_{B}$ are dcpo homomorphisms constructed as in Lemma \ref{dcpoequalizer}. It is clear how to extend this definition to morphisms and so we have defined a functor $f_*:\mathcal{F}\rTo\mathcal{E}$. By Lemma \ref{omegaXtoA}   $\mathcal{E}(A,\Omega_{\mathcal{E}}LB)\cong\mathbf{dcpo}_{\mathcal{E}}(\Omega_{\mathcal{E}}0,\Omega_{\mathcal{E}}(A \times LB))$
naturally in dcpo homomorphisms. So, since $\mathcal{E}(A, \Omega_{\mathcal{E}}LB) = \mathcal{E}(A, \overline{L}^{op} PB )$, we can prove that $f^* \dashv f_*$ by constructing a bijection
\begin{equation*}
\Lambda:\mathbf{dcpo}_{\mathcal{E}}(\Omega_{\mathcal{E}}0,\Omega_{\mathcal{E}}(A \times LW)) \\
\rTo \mathbf{dcpo}_{\mathcal{F}}(\Omega_{\mathcal{F}}0,\Omega_{\mathcal{F}}(RA \times W))
\end{equation*}
natural in dcpo homomorphisms between frames $\Omega_{\mathcal{F}}W$. Define $\Lambda$ by
\begin{eqnarray*}
\Lambda(\psi')&=&\Omega_{\mathcal{E}}( Id_{RA} \times \eta_W )\overline{R}^{op}(\psi')
\end{eqnarray*}
and this is clearly natural in dcpo homomorphisms between $\Omega_{\mathcal{F}}W$s by construction and by exploiting the assumption that $R$ preserves the strength and so preserves $q^X$ for any locale $X$ and any dcpo homomorphism $q$. To complete our construction of a geometric morphism we have to check that $\Lambda$ is a bijection.
However for any object $A$ of $\mathcal{E}$ consider the following diagram:

\begin{diagram}
\mathbf{Sup}_{\mathcal{E}}(\Omega_{\mathcal{E}}A,\Omega_{\mathcal{E}}(LW)) & \rTo^{\cong}& \mathbf{Sup}_{\mathcal{E}}(\Omega_{\mathcal{E}},\Omega_{\mathcal{E}}(A \times LW)) & \rTo^{\cong} & \mathbf{dcpo}_{\mathcal{E}}(\Omega_{\mathcal{E}}0,\Omega_{\mathcal{E}}(A \times LW)) \\
\dTo^{\overline{R}^{op}} & & \dTo^{\overline{R}^{op}} & &  \dTo^{\overline{R}^{op}}\\
\mathbf{Sup}_{\mathcal{F}}(\Omega_{\mathcal{F}}RA,\Omega_{\mathcal{F}}(RLW)) & \rTo^{\cong}& \mathbf{Sup}_{\mathcal{F}}(\Omega_{\mathcal{F}},\Omega_{\mathcal{F}}(RA \times RLW)) & \rTo^{\cong} & \mathbf{dcpo}_{\mathcal{F}}(\Omega_{\mathcal{F}}0,\Omega_{\mathcal{F}}(RA \times RLW)) \\
\dTo^{\Omega_{\mathcal{F}}\eta_{W} \circ ( \_ )}  & &  & &  \dTo^{\Omega_{\mathcal{F}}(Id_{RA} \times \eta_{W} ) \circ ( \_ )} \\
\mathbf{Sup}_{\mathcal{F}}(\Omega_{\mathcal{F}}RA,\Omega_{\mathcal{F}}(W)) & & \rTo^{\cong}& & \mathbf{dcpo}_{\mathcal{F}}(\Omega_{\mathcal{F}}0,\Omega_{\mathcal{F}}(RA \times W)) \\
\end{diagram}
Certainly the bottom rectangle commutes by construction. The top left square commutes by Lemma \ref{psi_bar_explicit} because $\overline{R}^{op}$ preserves the relevant structure. The top right hand square commutes because it can be readily checked that $\overline{R}^{op}$ preserves the top map $\{\_\}:\Omega0 \rTo \Omega$ which is the universal map from $1$ to the free suplattice on $1$ (i.e. $\Omega$). Therefore, because the left hand vertical map is a bijection (Lemma \ref{suplatticespeciallemma}), we have that the right hand vertical map, i.e. $\Lambda$, is a bijection and so $f^* \dashv f_*$ and we have defined a geometric morphism $f:\mathcal{F} \rTo \mathcal{E}$.

To check $L \cong \Sigma_f$ we observe, for any locale $W$ over $\mathcal{F}$ and for any poset $A$ of $\mathcal{E}$
\begin{eqnarray*}
\mathbf{Pos}_{\mathcal{E}}(A,\Omega_{\mathcal{E}}\Sigma_f(W)) & =  & \mathbf{Pos}_{\mathcal{E}}(A,f_* \Omega_{\mathcal{F}}W)\\
& \cong & \mathbf{Pos}_{\mathcal{F}}(f^*A,\Omega_{\mathcal{F}}W)\\
& = & \mathbf{Pos}_{\mathcal{F}}(RA,\Omega_{\mathcal{F}}W)\\
& \cong & \{ \psi' :  \Omega_{\mathcal{F}}0 \rTo \Omega_{\mathcal{F}} (RA \times W) | \uparrow_{RA}^W \psi' = \psi' \} \\
& \cong & \{ \psi' :  \Omega_{\mathcal{E}}0 \rTo \Omega_{\mathcal{E}} (A \times LW) | \uparrow_{A}^{LW} \psi' = \psi' \} \text{ \ via $\Lambda^{-1}$}\\
& \cong & \mathbf{Pos}_{\mathcal{E}}(A,\Omega_{\mathcal{E}}LW)\\
\end{eqnarray*}
where the second and third last lines are using Lemmas \ref{monotone} and \ref{closure_preserved}. Therefore, $L \cong \Sigma_f$.

Finally the uniqueness of $f:\mathcal{F}\rTo\mathcal{E}$ (up to natural transformation) is trivial: if $f':\mathcal{F}\rTo\mathcal{E}$ is some other geometric morphism with $(f')^*:\mathbf{Loc}_{\mathcal{E}}\rTo\mathbf{Loc}_{\mathcal{F}}$ isomorphic to $R$ then $(f')^*A \cong R(A) \cong f^*(A)$ for every discrete locale $A$ over $\mathcal{E}$ and so $f'\cong f$.
\end{proof}

\section{Dropping the strength}

In this section we check that (ii) and (iii) of the main theorem are equivalent categories. It can be seen that this follows from proving,

\begin{proposition}\label{Beck_Chevalley}
Given an order-enriched adjunction $L\dashv R:
\mathbf{Loc}_{\mathcal{F}}\pile{\rTo \\ \lTo} \mathbf{Loc}_{\mathcal{E}}$ and a monad isomorphism $\phi:R\mathbb{P}_{\mathcal{E}}\cong \mathbb{P}_{\mathcal{F}}R$ such that $R$ preserves finitary coproduct then $R$ also preserves the strength on $\mathbb{P}$ up to $\phi$.
\end{proposition}
To prove the Proposition, by Lemma \ref{characterizepreservingstrength}, what is required is a proof that $\overline{R}^{op}(q^Y) \cong [\overline{R}^{op}(q)]^{RY}$ for any dcpo homomorphism $q$ between frames in $\mathcal{E}$. In outline the proof exploits the fact that $q^Y$ can be described by using the lifting to Kleisli categories of localic change of base functors (\cite{towdcpo}) and then checking that $\overline{R}^{op}$ commutes with these lifted change of base functors. This is achieved by checking that the property of having $\phi$ is stable under slicing and then, effectively, verifying Beck-Chevalley for the lifting of the square that relates the slice of $L \dashv R$ back down to $L \dashv R$ via the change of base functors.

The next subsection introduces the necessary background material on weak triquotient assignments, the second subsection proves the necessary slice stability result. The final subsection consists of a proof of the Proposition. It is worth noting that the entire proof can be carried out axiomatically using the categorical description of weak triquotient assignments contained in \cite{towaxioms}; but for the sake of not having to increase the complexity of our exposition by introducing the relevant categorical axioms, our Proposition and proof is left a result about $\mathbf{Loc}$.

\subsection{Weak triquotient assignments}
Weak triquotient assignments will play a central role in the proof, essentially because they provide an external characterization of dcpo homomorphisms in each localic slice.

\begin{definition}
If $f:X\rTo Y$ is a map in the category of locales then a \emph{weak
triquotient assignment on }$f$ is a dcpo homomorphism $f_{\#}:\Omega
X\rTo \Omega Y$ with the property that
\begin{equation*}
f_{\#}(a_{1}\wedge \lbrack a_{2}\vee \Omega f(b)])=[f_{\#}(a_{1})\wedge
b]\vee f_{\#}(a_{1}\wedge a_{2})
\end{equation*}
for all $a_{1},a_{2}\in \Omega X$ and $b\in \Omega Y$.
\end{definition}

Weak triquotient assignments were originally isolated by Vickers in an
unpublished note as a weakening of Plewe's notion of triquotient assignment \cite{triquot}.
The notion is a localic form of the topological notion of triquotient map
introduced by Michael. Note that a
lattice theoretic manipulation shows that the defining equation is
equivalent to requiring that both
\begin{equation*}
\text{(a) \ \ \ \ \ \ \ }f_{\#}(a)\wedge b\leq f_{\#}(a\wedge \Omega f(b))
\end{equation*}
and
\begin{equation*}
\text{(b) \ \ \ \ \ \ \ }f_{\#}(a\vee \Omega f(b))\leq f_{\#}(a)\vee b
\end{equation*}
for all $a\in \Omega X$ and $b\in \Omega Y$. We shall use this characterization of weak triquotient assignment in our applications below.

Any open map $f:X\rTo Y$ gives rise to an example of a weak triquotient assignment since,
it can be verified by elementary lattice theoretic manipulations, $\exists_f$ is a weak triquotient assignment on $f$. Another class of examples arises from any split localic surjection: if $q: X \rTo Y$ is split by $i: Y \rTo X$ then $q_{\#}= \Omega i$ is a weak triquotient assignment on $q$.

Weak triquotient assignments are preserved by $\overline{R}^{op}$:
\begin{lemma}
\label{Prelemma}Given $R:\mathbf{Loc}_{\mathcal{E}}\rTo \mathbf{Loc}_{\mathcal{F}}$ preserving binary coproduct and for which there is a monad opfunctor $\phi: R \mathbb{P}_{\mathcal{E}} \rTo \mathbb{P}_{\mathcal{F}} R$, then for any weak triquotient assignment $f_{\#}:\Omega _{%
\mathcal{E}}X\rTo \Omega _{\mathcal{E}}X^{\prime }$ on a locale map $%
f:X\rTo X^{\prime }$ in $\mathcal{E}$, $\overline{R}^{op}(f_{\#})$ is
a weak triquotient assignment on $R(f):RX\rTo RX^{\prime }$.
\end{lemma}

\begin{proof}
Just as in the proof of Lemma \ref{suplatticespeciallemma}, $\overline{R}^{op}$ preserves frame meet and join and therefore preserves the equations that define weak triquotient assignment.
\end{proof}

Given a locale map $f:X \rTo Y$, there is a corresponding locale in the topos of sheaves over $Y$ under the equivalence $\mathbf{Loc}/Y \simeq \mathbf{Loc}_{Sh(Y)}$ (\cite{JoyT}). The key property of weak triquotient assignments on $f:X \rTo Y$ is that they are an external representation of the internal points of the double power of this corresponding locale in the topos of sheaves over $Y$. This allows us to translate facts about internal dcpo homomorphisms relative to sheaves over $Y$ into facts about dcpo homomorphisms in the relevant ambient topos. Since the points of the double power locale are in order isomorphism with dcpo homomorphisms, this key property can be expressed as:
\begin{lemma}\label{dcpo_at_one}
There is an order isomorphism between weak triquotient assignments on $%
f:X\rTo Y$ and internal dcpo homomorphisms,
\begin{equation*}
\mathbf{dcpo}_{Y}(\Omega _{Y}X_{f},\Omega _{Y})\text{.}
\end{equation*}
\end{lemma}
Note that we are using $\mathbf{dcpo}_{Y}$ to denote the category of dcpos relative to the topos, $Sh(Y)$, of sheaves over $Y$. $X_f$ is used as the object of $\mathbf{Loc}/Y$ given by the morphism $f : X \rTo Y$ and $X_Y$ is used for the object $X \times Y \rTo^{\pi_2} Y$. Every locale map $f: X \rTo Y$ gives rise to a geometric morphism $f: Sh (X) \rTo Sh (Y) $ and so by (i) implies (ii) of the main theorem there is an order-enriched adjunction $ (\mathbf{Loc}_{X})^{op}_{\mathbb{P}_{X}}\pile{\rTo \\ \lTo} (\mathbf{Loc}_{Y})^{op}_{\mathbb{P}_{Y}}$ which is a contravariant lifting of the pullback adjunction $\Sigma_f \dashv f^* : \mathbf{Loc}/X \pile{\rTo \\ \lTo } \mathbf{Loc}/Y$; we use $f^{\#} \dashv f_*$ for this order-enriched adjunction. Note that the unit of this order-enriched adjunction at $\Omega_Y X'_{f'}$ is $\Omega_Y X'_{f'} \rTo^{\Omega_Y \pi_2} \Omega_Y [X \times_Y X']_{f\pi_1} $ and the counit at $\Omega_X W_g$ is $\Omega_X [X \times_Y W]_{\pi_1} \rTo^{\Omega_X (g,Id_W)} \Omega_X W_g$. The term \emph{change of base} will be used when we are passing through this adjunction. Finally, we will use $\gamma^Y$ for the unique geometric morphism from $Sh(Y)$ back to the relevant base topos.
\begin{proof}
Consult Lemma 41 in \cite{towdcpo}. The order isomorphism, in one direction, sends any dcpo homomorphism $q:\Omega
_{Y}X_{f}\rTo \Omega _{Y}$ to $\gamma^Y _{\ast }(q)$. In the other direction for any weak triquotient assignment $f_{\#} : \Omega X \rTo \Omega Y$ of $f$, the corresponding dcpo homomorphism relative to $Sh(Y)$ is the unique $q: \Omega_Y X_f \rTo \Omega_Y$ such that $\Omega_Y X_Y \rTo^{ \Omega_Y ( Id_X,f)  } \Omega_Y X_f  { \rTo ^q}\Omega_Y$ is equal to $\Omega_Y X_Y \rTo^{\widetilde{f_{\#}}} \Omega_Y$ where $\widetilde{(\_)}$ is adjoint transpose with respect to $(\gamma^Y)^{\#} \dashv \gamma^Y_*$.
\end{proof}
As an application we prove the key pullback stability result for locale maps with weak triquotient assignments:
\begin{lemma}
\label{pullback}
Given a pullback square
\begin{diagram}
X\times _{Y}X^{\prime } & \rTo^{\pi _{1}} & X \\
\dTo^{\pi _{2}} &  & \dTo^f \\
X^{\prime } & \rTo^{{f^{\prime }}} & Y
\end{diagram}
in $\mathbf{Loc}$ and any weak triquotient assignment $f_{\#}:\Omega
X\rTo \Omega Y$ on $f$ there exists a unique weak triquotient
assignment $(\pi _{2})_{\#}:\Omega (X\times _{Y}X^{\prime })\rTo
\Omega X^{\prime }$ on $ \pi_2: X\times _{Y}X^{\prime }\rTo X^{\prime }$ such
that the \emph{Beck-Chevalley condition} holds; i.e. such that $(\pi
_{2})_{\#}\Omega \pi _{1}=\Omega {f^{\prime}}f_{\#}$.
\end{lemma}
This was proved originally by Plewe, \cite{triquot}, for triquotient assignments and the proof for weak triquotient assignments follows a similar path. The result was proved by Vickers in the unpublished note that originally isolated the class of weak triquotient assignments. The existence part is proved in \cite{towdcpo} using the representation theorem of the previous lemma and we repeat the proof here and include the uniqueness part.
\begin{proof}
If $q: \Omega_Y (X_f) \rTo \Omega_Y$ is the dcpo homomorphism relative to $Sh(Y)$ that corresponds to $f_{\#}$ then $(f^{\prime})^{\#}(q)$ is a dcpo homomorphism from $\Omega_{X'} ([X \times_Y X']_{\pi_2})$ to $ \Omega_{X'}$ relative to $Sh(X')$ and so gives rise to a weak triquotient assignment $\gamma^{X'}_* ((f^{\prime})^{\#}(q))$ on $\pi_2$.

The map $\Omega X \rTo^{f_{\#}} \Omega Y \rTo^{\Omega f'} \Omega X'$ is equal to
\begin{eqnarray*}
\Omega_Y X_Y \rTo^{\Omega_Y (Id_X,f)} \Omega_Y X_f \rTo^q \Omega_Y \rTo^{\Omega_Y f'} \Omega_Y X'_{f'}
\end{eqnarray*}
after passing through the adjunction $(\gamma^Y)^{\#} \dashv \gamma_*^Y$ (i.e. by changing base to $Y$). This is because $\gamma_*^Y$ takes $\Omega_Y f'$ to $\Omega f'$ and the definition of $q$. Since $\Omega_Y \rTo^{\Omega_Y f'} \Omega_Y X'_{f'}$ is the unit of $(f')^{\#} \dashv f'_*$ at $\Omega_Y$, if we next change base along $f'$ we obtain $ \Omega_{X'} X_{X'} \rTo^{\Omega_{X'} (i)} \Omega_{X'} [X \times_Y X']_{\pi_2} \rTo^{(f')^{\#} (q)} \Omega_{X'}$ where $i$ is the inclusion $X \times_Y X' \rInto X \times X'$.

On the other hand, since the adjoint transpose of $\Omega X \rTo^{\pi_1} \Omega (X \times_Y X')$ under the adjunction $(\gamma^{X'})^{\#} \dashv \gamma^{X'}_*$ is $\Omega_{X'} (i)$ we have that, under change of base to $X'$, $\Omega X \rTo^{\Omega \pi_1} \Omega(X \times_Y X') \rTo^{\gamma^{X'}_* ((f')^{\#}(q))} \Omega X'$ is equal to $ \Omega_{X'} X_{X'} \rTo^{\Omega_{X'} (i)} \Omega_{X'} [X \times_Y X']_{\pi_2} \rTo^{(f')^{\#} (q)} \Omega_{X'}$ and so this completes a verification of the Beck-Chevalley condition given in the statement of the lemma.

For the uniqueness part of the statement of the corollary notice that if $(\pi_2)'_{\#}$ is some other weak triquotient assignment such that  $(\pi_2)'_{\#} \Omega \pi_1 = \Omega f' f_{\#}$ then by change of base to $X'$,
\begin{eqnarray*}
\Omega_{X'} X_{X'} \rTo^{\Omega_{X'} (i)} \Omega_{X'} (X \times_Y X') \rTo^{q'} \Omega_{X'} = \\
\Omega_{X'} X_{X'} \rTo^{\Omega_{X'} (i)} \Omega_{X'} (X \times_Y X') \rTo^{(f')^{\#}(q)} \Omega_{X'} \\
\end{eqnarray*}
where $q'$ is the dcpo homomorphism relative to $Sh(X')$ that corresponds to $(\pi_2)'_{\#}$. It follows that $q'=(f')^{\#}(q)$ and therefore that $(\pi_2)'_{\#}= (\pi_2)_{\#}$, because $\Omega_{X'}(i)$ is an epimorphism since $i$ is a regular monomorphism.
\end{proof}
\begin{corollary}
\label{splitting}
Given the conditions of the lemma, if $f:X \rTo Y$ is a split surjection (split by $i: Y \rTo X$) then the unique weak triquotient assignment on $\pi_2$ corresponding to $f_{\#} = \Omega i$ is $\Omega (\pi_1^* (i))$ where $\pi_1^* (i)$ is the pullback of $i$ along $\pi_1 : X \times_Y X' \rTo X$.
\end{corollary}
\begin{proof}
The adjoint transpose of $\Omega X \rTo^{\Omega i} \Omega Y$ across $(\gamma^Y)^{\#} \dashv \gamma_*^Y$ is $\Omega_Y X_Y \rTo^{\Omega_Y (i,Id_Y)} \Omega_Y$ which factors as $\Omega_Y X_Y \rTo^{ (Id_X, f)} \Omega_Y X_f \rTo^{\Omega_Y i} \Omega_Y$. Therefore the dcpo homomorphism relative to $Sh(Y)$ corresponding to the weak triquotient assignment $\Omega i$ on $f$ is equal to $\Omega_X X_f \rTo^{\Omega_Y i} \Omega_Y$. By considering the pullback diagrams
\begin{diagram}
X' & \rTo^{f'} & Y \\
\dTo^{(if',Id_{X'})} &  & \dTo^i \\
X \times_Y X' & \rTo^{\pi_1} & X \\
\dTo^{\pi_2} &   & \dTo^f \\
X' & \rTo^{f'} & Y
\end{diagram}
it can be seen that $(f')^{\#}(\Omega_Y i)$ is $\Omega_{X'} (X \times_Y X')_{\pi_2} \rTo^{\Omega_{X'}(if',Id_{X'})} \Omega X'$, i.e. $\Omega_{X'}(\pi_1^*(i))$. The result follows since $\gamma^{X'}_*(\Omega_{X'}(\pi_1^*(i)))=\Omega(\pi_1^*(i))$.
\end{proof}
Note that because Lemma \ref{dcpo_at_one} can be proved using only topos valid reasoning it can be carried
out in the topos of sheaves $Sh(Z)$ for any locale $Z$ and so we have:
\begin{lemma}
\label{wtqa_dcpohoms}For any locale maps $f:X\rTo Y$ and $f^{\prime
}:X^{\prime }\rTo Y$ there is an order isomorphism between dcpo
homomorphisms $q:\Omega _{Y}X_{f}\rTo \Omega _{Y}X_{f^{\prime
}}^{\prime }$ and weak triquotient assignments on $\pi _{2}:X\times
_{Y}X^{\prime }\rTo X^{\prime }$.
\end{lemma}

\begin{proof}
By change of base to $Sh(X^{\prime })$. The geometric morphism $f^{\prime
}:Sh(X^{\prime })\rTo Sh(Y)$ induces an order isomorphism between $%
\mathbf{dcpo}_{Y}(\Omega _{Y}X_{f},\Omega _{Y}X_{f^{\prime
}}^{\prime })$ and $\mathbf{dcpo}_{X^{\prime }}(\Omega _{Y}([X\times
_{Y}X^{\prime }]_{\pi _{2}}),\Omega _{X^{\prime }})$.
\end{proof}

The notation $a^q: \Omega (X \times_Y X') \rTo \Omega X' $ is used for the weak triquotient assignment corresponding to $q:\Omega _{Y}X_{f}\rTo \Omega _{Y}X_{f^{\prime }}^{\prime }$. For example, for any locale map $h: X'_{f'} \rTo X_f$ in the slice $\mathbf{Loc}/Y$, $a^{\Omega_Y(h)}$ is equal to $\Omega (X \times_Y X') \rTo^{\Omega (h,Id_{X'})} \Omega X'$. Since $q$ factors as $\Omega_Y X_f \rTo^{\Omega_Y \pi_1} \Omega_Y [X \times_Y X']_{f \pi_1} \rTo^{f'_*(\hat{q})} \Omega_Y X'_{f'}$ by change of base across $(f')^{\#} \dashv f'_*$ (where $\hat{(\_)}$ is adjoint transpose with respect to $(f')^{\#} \dashv f'_*$), we have that $\gamma_*^Y(q)$ is equal to $\Omega X \rTo^{\pi_1} \Omega(X \times_Y X') \rTo^{a^q} \Omega X$. This is usually enough to establish facts about $q$ from $a^q$ since $\gamma_*^Y$ is faithful as the counit of $(\gamma^Y)^{\#} \dashv \gamma^Y_*$ is equal to $\Omega_Y (Id_X ,f)$ which is an epimorphism since $(Id_X,f): X_f \rInto X_Y$ is a regular monomorphism.

Our final technical step in setting up the necessary background on weak triquotient assignments is to check that the previous lemma is natural in locale maps:
\begin{lemma}\label{Naturality}
Given an internal dcpo homomorphism $q:\Omega
_{Y}X_{f}\rTo \Omega _{Y}X_{f^{\prime }}^{\prime }$,

(i) For any $h:X_{f^{\prime \prime }}^{\prime
\prime }\rTo X_{f^{\prime }}^{\prime }$, $a^{\Omega_Y (h) q}$ is the unique weak triquotient assignment on $\pi_2^{\prime \prime} : X\times _{Y}X^{\prime \prime }  \rTo X^{ \prime \prime}$ such that Beck-Chevalley holds for the
pullback square
\begin{diagram}
X\times _{Y}X^{\prime \prime } & \rTo^{Id_X\times h} &
X\times _{Y}X^{\prime } \\
\dTo^{\pi _{2}^{\prime \prime }} &  & \dTo^{\pi _{2}^{\prime}} \\
X^{\prime \prime } & \rTo^{h} & X^{\prime }
\end{diagram}
i.e. such that $a^{\Omega_Y (h) q}\Omega (Id_X \times h)=\Omega(h) a^q$.

(ii) For any morphism $t:X_f \rTo Z_g$ of $\mathbf{Loc}/Y$, $a^{q \Omega_Y t} = a^q \Omega (t \times Id_{X'})$ with $t \times Id_{X'}: X \times_Y X' \rTo Z \times_Y X'$.
\end{lemma}

\begin{proof}
{\bf (i)}. Let us first consider $h: X^{\prime \prime} \rTo X'$ as a geometric morphism and note that the unit of $h^{\#} \dashv h_*$ at $\Omega_{X'}([X \times_Y X']_{\pi'_2})$ is
\begin{eqnarray*}
\Omega_{X'} ( [X \times_Y X']_{\pi'_2}) \rTo^{\Omega_{X'} (Id_X \times h)} \Omega_{X'}( [X \times_Y X^{\prime \prime}]_{h \pi_2^{\prime \prime}})\text{.}
\end{eqnarray*}
This can be seen, for example, by unraveling the isomorphism $X \times_Y X^{\prime \prime} \cong [ X \times_Y X']_{\pi_2} \times_{X'} X^{\prime \prime}$.

Since $a^{\Omega_Y(h)q}$ is by definition a weak triquotient assignment on $\pi_2^{\prime \prime}: X \times_Y X^{\prime \prime} \rTo X^{\prime \prime}$ all that is required to complete the proof of (i) is to verify that $a^{\Omega_Y (h) q}\Omega (Id\times h)=\Omega(h) a^q$. Now, $a^{\Omega_Y(h)q} = \gamma^{X^{\prime \prime}}_*(\widehat{ \Omega_Yh q})$, where $\widehat{(\_)}$ is change of base to $X^{\prime \prime}$. But $\widehat{(\_)}$ can be found by first changing base to $X'$ and then changing base via $h$. The image of $\Omega_Y (h) q$ relative to $X'$ is
\begin{eqnarray*}
\Omega_{X'}( [X \times_Y X']_{\pi'_2}) \rTo^{\widetilde{q}} \Omega_{X'} \rTo^{\Omega_{X'} h} \Omega_{X'} X^{\prime \prime}_h
\end{eqnarray*}
where $\widetilde{(\_)}$ is change of base to $X'$. The image of this, via $h$, is $\Omega_{X^{\prime \prime}}([X \times_Y X^{\prime \prime}]_{\pi_2^{\prime \prime}}) \rTo^{h^{\#}(\widetilde{q})} \Omega_{X^{\prime \prime}}$ because $\Omega_{X'} h$ is the unit of $h^{\#} \dashv h_*$ at $\Omega_{X'}$. It follows:
\begin{eqnarray*}
a^{\Omega_Y(h)q}\Omega(Id_X \times h)   & = & \gamma^{X^{\prime \prime}}_*(\widehat{\Omega_Yh q})\Omega(Id_X \times h) \\
                                        & = & \gamma^{X^{\prime \prime}}_*h^{\#}(\widetilde{q})\Omega(Id_X \times h) \\
                                        & = &  \gamma^{X^{\prime}}_* h_* h^{\#}(\widetilde{q})\gamma^{X^{\prime}}_* \Omega_{X'}(Id_X \times h) \\
                                        & = &  \gamma^{X^{\prime}}_* (h_* h^{\#}(\widetilde{q}) \Omega_{X'}(Id_X \times h)) \\
                                        & = &  \gamma^{X^{\prime}}_* (\Omega_{X'}(h) (\widetilde{q})) \text{ \ \ \ \ \ \ \ \ \ \ \ \ \ \ \ \ \ \ \ \ \ \      ($\star$})\\
                                        & = &  \gamma^{X^{\prime}}_* \Omega_{X'}(h) \gamma^{X^{\prime}}_*(\widetilde{q}) \\
                                        & = & \Omega(h) a^q
\end{eqnarray*}
where we are exploiting our first observation that $\Omega_{X'}(Id_X \times h)$ is a unit at stage $(\star)$.

{\bf (ii)}. $q \Omega_Y t$ is equal to $\Omega_{X'} ( [Z \times_Y X']_{\pi'_2}) \rTo^{\Omega_{X'}(t \times Id_{X'})} \Omega_{X'} ([X \times_Y X']_{\pi'_2}) \rTo^{\widetilde{q}} \Omega_{X'}$ when changed to base $X'$, where $\widetilde{q}$ is the mate of $q$. The result follows by applying $\gamma^{X'}_*$ since $a^q = \gamma_*^{X'}( \widetilde{q})$ and $ \gamma_*^{X'}(\Omega_{X'} (t \times Id_{X'}) ) = \Omega (t \times Id_{X'})$.
\end{proof}

\subsection{Slice stability}
We now embark on a series of lemmas that culminates in showing that (iii) of the main theorem is stable under slicing. If $L\dashv R$ is an order-enriched adjunction between $\mathbf{Loc}_{\mathcal{F}}$ and $\mathbf{Loc}_{\mathcal{E}}$ then for any locale $Y$ of $\mathcal{E}$ there is a sliced order-enriched adjunction,
\begin{equation*}
\mathbf{Loc}_{\mathcal{F}}/RY
\begin{array}{c}
\overset{L_{Y}}{\rTo } \\
\underset{R_{Y}}{\lTo }
\end{array}
\mathbf{Loc}_{\mathcal{E}}/Y
\end{equation*}
given by $R_{Y}(X_{f})=R(X)_{R(f)}$ and $L_{Y}(W_{g})=LW_{\widetilde{g}}$,
where $\widetilde{g}$ is the mate of $g$ under the adjunction $L\dashv R$. It is easy to see that if $R$ preserves finitary coproduct then so does $R_Y$ because finitary coproducts in any slice $\mathbf{Loc}/Z$ are created in $\mathbf{Loc}$ (this is true of any category). To show that the conditions in (iii) of the main theorem hold for the sliced adjunction, given an assumption that they hold for $L\dashv R$, we therefore need to show that there is a lifting $  \overline{L_Y} \dashv \overline{R_Y}$ to the Kleisli categories given a lifting $  \overline{L} \dashv \overline{R}$ of $L \dashv R$; this is by appeal to Lemma \ref{adjunction lemma}. In fact we construct the contravariant adjunction $  \overline{R_Y}^{op} \dashv \overline{L_Y}^{op}$ and our first step is to establish an order isomorphism between the relevant sets of dcpo homomorphisms:

\begin{lemma}\label{define_tau}
There is an order isomorphism
\begin{eqnarray*}
\tau_{X_f,W_g}:\mathbf{dcpo}_{Y}(\Omega_{Y}X_{f},\Omega _{Y}L_{Y}(W_{g})) \rTo^{\cong} \mathbf{dcpo}_{RY}(\Omega _{RY}R_{Y} (X_{f}),\Omega _{RY}W_{g})\text{.}
\end{eqnarray*}
for any objects $X_f$ and $W_g$ of $\mathbf{Loc}/Y$ and $\mathbf{Loc}/RY$ respectively. If $n: L_Y(W_g) \rTo X_f$ is a morphism of $\mathbf{Loc}/Y$ then $\tau_{X_f,W_g} \Omega_Y(n) = \Omega_{RY} \widetilde{n}$ where $\widetilde{(\_)}$ is adjoint transpose with respect to $L_Y \dashv R_Y$.
\end{lemma}
\begin{proof}
Given the order isomorphism between such dcpo homomorphisms and weak
triquotient assignments (Lemma \ref{wtqa_dcpohoms}) this amounts to
establishing an order isomorphism between weak triquotient assignments. Given a dcpo homomorphism $q:\Omega_Y X_f \rTo \Omega_Y L_Y (W_g)$ there is $a^q$ the corresponding weak triquotient assignment on $X\times _{Y}LW\overset{%
\pi _{2}}{\rTo }LW$. Then $\overline{R}^{op}(a^q)$ is a weak
triquotient assignment on $RX\times _{RY}RLW\overset{\pi _{2}}{\rTo }%
RLW$ by Lemma \ref{Prelemma} and so by the pullback property of weak
triquotient assignments, since
\begin{diagram}
RX\times _{RY}W & \rTo^{Id_{RX}\times \eta _{W}} & RX\times
_{RY}RLW \\
\dTo^{\pi _{2}} &  & \dTo^{\pi _{2}} \\
W & \rTo^{\eta _{W}} & RLW
\end{diagram}
is a pullback square, there exists a unique weak triquotient assignment, $%
b :$ $\Omega _{\mathcal{F}}(RX\times _{RY}W)\rTo \Omega _{%
\mathcal{F}}(W)$ say, such that
\begin{equation*}
b \Omega _{\mathcal{F}}(Id_{RX}\times \eta _{W})=\Omega _{\mathcal{F}%
}\eta _{W}\overline{R}^{op}(a^q)\text{ \ \ \ \ \ Eqn. I.}
\end{equation*}
Define $\tau_{X_f,W_g}(q)$ by $b = a^{\tau_{X_f,W_g}(q)}$; i.e. $\tau_{X_f,W_g}(q)$ is the unique dcpo homomorphism over $Sh_{\mathcal{F}}(RY)$ whose corresponding weak triquotient assignment is $b$.

In the other direction if $a^r$ is the weak triquotient assignment
on $RX\times _{RY}W\overset{\pi _{2}}{\rTo }W$ corresponding to some dcpo homomorphism $r: \Omega_{RY} R_Y(X_f) \rTo \Omega_{RY} W_g$ then define $c : \Omega _{\mathcal{F}}(X\times _{Y}LW)\rTo \Omega _{\mathcal{F}}(LW)$
to be the adjoint transpose of the map
\begin{equation*}
\Omega _{\mathcal{F}}(RX\times _{RY}RLW)\rTo^{\Omega _{\mathcal{F}%
}(Id_{RX}\times \eta _{W})}\Omega _{\mathcal{F}}(RX\times
_{RY}W)\rTo^{a^r}\Omega _{\mathcal{F}}(W)
\end{equation*}
with respect to the adjunction $\overline{R}^{op} \dashv \overline{L}^{op}$. Consider the diagram,
\begin{diagram}
\Omega _{\mathcal{F}}(RX {\times _{RY}}RLW){\times} \Omega _{\mathcal{F}}(RLW) & \rTo^{Id\times \overline{R}^{op}\Omega _{\mathcal{E}}\pi _{2}}
 & \Omega _{\mathcal{F}}(RX\times _{RY}RLW)\times \Omega _{\mathcal{F}}(RX\times _{RY}RLW) & \rTo^{\overline{R}^{op}\wedge } & \Omega _{\mathcal{F}}(RX\times _{RY}RLW) \\
\dTo_{\Omega _{\mathcal{F}}(Id_{RX}\times \eta _{W})\times \Omega _{\mathcal{F}}(\eta _{W})} &  &  &  & \dTo^{\Omega _{\mathcal{F}}(Id_{RX}\times \eta_{W})} \\
\Omega _{\mathcal{F}}(RX\times _{RY}W)\times \Omega _{\mathcal{F}}(W) &
\rTo^{Id\times \Omega _{\mathcal{F}}\pi _{2}} & \Omega
_{\mathcal{F}}(RX\times _{RY}W)\times \Omega _{\mathcal{F}}(RX\times _{RY}W)
& \rTo^{\wedge } & \Omega _{\mathcal{F}}(RX\times
_{RY}W) \\
\dTo_{a^r \times Id} &  &  &  & \dTo^{a^r} \\
\Omega _{\mathcal{F}}(W)\times \Omega _{\mathcal{F}}(W) &  & \rTo{\wedge
_{\Omega _{\mathcal{F}}(W)}} &  & \Omega _{\mathcal{F}}(W)
\end{diagram}

The top square commutes because $\overline{R}^{op}(\wedge )$ is frame meet (since $R$ preserves the co-diagonal and $\overline{R}^{op}$ is order-enriched)
and $\overline{R}^{op}\Omega _{\mathcal{E}}\pi _{2} \cong \Omega _{\mathcal{F}}\pi _{2}$. Since further we have that
\begin{equation*}
\wedge _{\Omega _{\mathcal{F}}(W)}(a^r \times Id)\leq a^r \wedge _{\Omega _{\mathcal{F}}(RX\times _{Y}W)}(Id\times \Omega _{%
\mathcal{F}}(\pi _{2}))
\end{equation*}
by definition of weak triquotient assignment, it follows that the bottom and
left hand vertical composition in this diagram is less than or equal to the
right hand and top composition, i.e. that
\begin{eqnarray*}
&&\wedge _{\Omega _{\mathcal{F}}(W)}(a^r\times Id)(\Omega _{%
\mathcal{F}}(Id_{RX}\times \eta _{W})\times \Omega _{\mathcal{F}}(\eta _{W}))
\\
&\leq &a^r\Omega _{\mathcal{F}}(Id_{RX}\times \eta _{W})(%
\overline{R}^{op}\wedge _{\Omega _{\mathcal{E}}(X\times _{Y}LW)})(Id\times
\overline{R}^{op}\Omega _{\mathcal{E}}\pi _{2})\text{.}
\end{eqnarray*}
The adjoint transpose (under $\overline{R}^{op}\dashv \overline{L}^{op}$) of
the top and right hand composition is $c \wedge _{\Omega _{\mathcal{E}%
}(X\times _{Y}LW)}(Id\times $ $\Omega _{\mathcal{E}}\pi _{2})$ since $%
Id\times \overline{R}^{op}\Omega _{\mathcal{E}}\pi _{2}\cong\overline{R}%
^{op}(Id\times \Omega _{\mathcal{E}}\pi _{2})$. Since we
also have that $\overline{L}^{op}\wedge _{\Omega _{\mathcal{F}%
}W}\cong\wedge _{\Omega _{\mathcal{E}}LW}$ the adjoint transpose of the bottom
and left hand vertical composition is $\wedge _{\Omega _{\mathcal{E}%
}LW}(c \times Id)$. Since the adjunction $\overline{R}^{op}\dashv
\overline{L}^{op}$ is order-enriched it establishes an order isomorphism on
homsets and so
\begin{equation*}
\wedge _{\Omega _{\mathcal{E}}LW}(c \times Id)\leq c \wedge
_{\Omega _{\mathcal{E}}(X\times _{Y}LW)}(Id\times \Omega _{\mathcal{E}}\pi
_{2})\text{.}
\end{equation*}
A dual argument, with binary meet in place of binary join establishes
\begin{equation*}
c \vee _{\Omega _{\mathcal{E}}(X\times _{Y}LW)}(Id\times \Omega _{%
\mathcal{E}}\pi _{2})\leq \vee _{\Omega _{\mathcal{E}}LW}(c \times Id)%
\text{,}
\end{equation*}
and so $c $ is a weak triquotient assignment on $\pi _{2}:X\times
_{Y}LW\rTo LW$.

By applying $\overline{R}^{op}$ to $c $ and postcomposing with $\Omega
_{\mathcal{F}}(\eta _{W})$ we get the adjoint transpose of $c $ (i.e. $%
a^r \Omega _{\mathcal{F}}(Id_{RX}\times \eta _{W})$). Therefore
the weak triquotient assignment on $\pi _{2}:RX\times _{Y}W\rTo W$
obtained from $c $ is again $a^r$ by the uniqueness of weak
triquotient assignments satisfying Beck-Chevalley for the pullback square,
\begin{diagram}
RX\times _{RY}W & \rTo^{Id_{RX}\times \eta _{W}} & RX\times
_{RY}RLW \\
\dTo^{\pi _{2}} &  & \dTo^{\pi _{2}} \\
W & \rTo^{\eta _{W}} & RLW
\end{diagram}
On the other hand given any weak triquotient assignment $a^q$ on
$\pi _{2}:X\times _{Y}LW\rTo LW$, the adjoint transpose of $b
\Omega _{\mathcal{F}}(Id_{RX}\times \eta _{W})$ ($b$ defined from $a^q$ as in before Eqn I above) is $a^q$ by application
of Eqn I. We have therefore established an order isomorphism between weak
triquotient assignments on $\pi _{2}:X\times _{Y}LW\rTo LW$ and weak
triquotient assignments on $\pi _{2}:RX\times _{Y}W\rTo W$, and
therefore, by application of Lemma \ref{wtqa_dcpohoms}, we have established an
order isomorphism as required.

For the assertion $\tau_{X_f,W_g} \Omega_Y(n) = \Omega_{RY} (\widetilde{n})$ note that because $a^{\Omega_Y(n)} = \Omega (n, Id_{LW})$ we have that $\overline{R}^{op}(a^{\Omega_Y(n)}) = \Omega (Rn, Id_{RL(W)})$. Now $RLW \rTo^{(Rn,Id_{RL(W)})} RX \times_{RY} RLW$ is a splitting of the localic surjection $RX \times_{RY} RLW \rTo^{\pi_2} RLW$ and so by Corollary \ref{splitting} the unique weak triquotient assignment corresponding to $\tau_{X_f,W_g} \Omega_Y(n)$ is equal to $\Omega (\widetilde{n},Id_W) $ because there is a pullback square:

\begin{diagram}
W & \rTo^{\eta_W} & RLW \\
\dTo^{(\widetilde{n},Id_W)} &  & \dTo^{(Rn,Id_{RLW})} \\
RX \times_Y W & \rTo^{Id_{RX} \times \eta_W } & RX \times_{RY} RLW
\end{diagram}
It follows that $\tau_{X_f,W_g} \Omega_Y(n) = \Omega_{RY} (\widetilde{n})$ because $a^{\Omega_{RY}(\widetilde{n})} = \Omega (\widetilde{n}, Id_{W})$ and so $\Omega_{RY}(\widetilde{n})$ is the unique dcpo homomorphism corresponding to the weak triquotient assignment $\Omega (\widetilde{n}, Id_{W})$.
\end{proof}
The next step is to make $\overline{R_Y}^{op}$ functorial. It should be clear how $\overline{R_Y}^{op}$ is going to be defined on morphisms given that $\tau_{X_f,W_g}$ is a contravariant extension of the order isomorphism $\mathbf{Loc}/Y (L_Y(W_g),X_f) \cong \mathbf{Loc}/RY(W_g,R_Y(X_f))$ induced by $L_Y \dashv R_Y$.
\begin{definition}
For any dcpo homomorphism $q: \Omega_Y X_f \rTo \Omega_Y X'_{f'}$ define $\overline{R_Y}^{op}(q)$ to be
\begin{eqnarray*}
\tau_{X_f,R_Y(X'_{f'})}(\Omega_Y X_f \rTo^{q} \Omega_Y X'_{f'} \rTo^{\Omega_Y \epsilon_{X'}} \Omega_Y L_Y R_Y X'_{f'})\text{.}
\end{eqnarray*}
\end{definition}
With this definition it is clear that $\overline{R_Y}^{op}(\Omega_Y h)= \Omega_{RY} R_Y (h)$ for any morphism $h$ of $\mathbf{Loc}/Y$ and so $\overline{R_Y}^{op}$ extends $R_Y$ contravariantly.
\begin{lemma}\label{R_Y_Functorial}
For any $q: \Omega_Y X_f \rTo \Omega_Y X'_{f'}$, (i) $\overline{R}^{op}(a^q) = a^{\overline{R_Y}^{op}(q)}$ and (ii) $\gamma^{RY}_* \overline{R_Y}^{op}(q)=  \overline{R}^{op} \gamma^Y_* (q)$. It follows that,

{\bf(a)} $\overline{R_Y}^{op}$ is functorial,

{\bf(b)} for $X'_{f'}=L_Y W_g$, $\tau_{X_f,W_g}(q) = \Omega_{RY} \eta_W \overline{R_Y}^{op}(q)$; and,

{\bf(c)} $\tau_{X_f,W_g}$ is natural in dcpo homomorphisms between $\Omega_Y X_f$s.
\end{lemma}
\begin{proof}
For (i), by application of the first part of the naturality lemma (Lemma \ref{Naturality}), we have that $a^{\Omega_Y \epsilon_{X'}q}$ is equal to the unique weak triquotient assignment, $b$, on $\pi_2 : X \times_Y L R X' \rTo LRX' $ such that $b \Omega_{\mathcal{E}}(Id_X \times \epsilon_{X'}) = \Omega_{\mathcal{E}} \epsilon_{X'} a^q$. Now consider the pullback diagrams
\begin{diagram}
RX \times_{RY} RX' & \rTo^{Id_{RX} \times \eta_{RX'}} & RX \times_{RY} RLRX' & \rTo^{Id_{RX} \times R \epsilon_{X'}} & RX \times_{RY} RX' \\
\dTo^{\pi_2}       &                                 &      \dTo^{\pi_2}     &                                     & \dTo^{\pi_2}         \\
RX'                & \rTo^{\eta_{RX'}}                &  RLRX'                & \rTo^{R \epsilon_{X'}}              & RX' \\
\end{diagram}
The right hand side square is the image under $R$ of the pullback diagram that is used to define $b$. By definition $a^{\overline{R_Y}^{op}(q)}$ is the unique weak triquotient assignment on $\pi_2 : RX \times_Y RX' \rTo RX'$ such that $a^{\overline{R_Y}^{op}(q)}\Omega_{\mathcal{F}}(Id_{RX} \times \eta_{RX'}) = \Omega_{\mathcal{F}} \eta_{RX'} \overline{R}^{op}(b)$. Therefore $\overline{R}^{op}(a^q) = a^{\overline{R_Y}^{op}(q)}$ by the uniqueness part of Lemma \ref{pullback} applied to the whole diagram since $\overline{R}^{op}(b)\Omega (Id_{RX} \times R \epsilon_{X'})= \Omega R \epsilon_{X'} \overline{R}^{op}(a^q)$ and both the bottom and top rows are identity maps.

For (ii), as was clarified after Lemma \ref{wtqa_dcpohoms}, we have that $\gamma_*^Y(q)$ factors as
\begin{eqnarray*}
\Omega_{\mathcal{E}} X \rTo^{\pi_1} \Omega_{\mathcal{E}} (X \times_{Y} X' )\rTo^{a^q} \Omega_{\mathcal{E}} X'\text{.}
\end{eqnarray*}
So (ii) follows from (i) since $\gamma_*^{RY}(\overline{R_Y}^{op}(q))$ factors as
\begin{eqnarray*}
\Omega_{\mathcal{F}} RX \rTo^{\pi_1} \Omega_{\mathcal{F}} (RX \times_{RY} RX') \rTo^{a^{\overline{R_Y}^{op}(q)}} \Omega_{\mathcal{F}} RX' \text{.}
\end{eqnarray*}
{\bf (a)} is immediate from (ii) since $\gamma^{RY}_*$ is faithful. For {\bf (b)}, $ a^{\tau_{X_f,W_g}(q)}$ is the unique weak triquotient assignment on $\pi_2:RX \times_{RY} W \rTo W$ such that $a^{\tau_{X_f,W_g}(q)} \Omega_{\mathcal{F}}(Id_{RX} \times \eta_W) = \Omega_{\mathcal{F}} \eta_W \overline{R}^{op} (a^q)$. By the naturality lemma $a^{\Omega_{\mathcal{F}}\eta_W \overline{R_Y}^{op}(q)}$ is the unique weak triquotient assignment on $\pi_2$ such that $a^{\Omega_{\mathcal{F}}\eta_W \overline{R_Y}^{op}(q)} \Omega_{\mathcal{F}}(Id_{RX} \times \eta_W) = \Omega_{\mathcal{F}} \eta_W  a^{ \overline{R_Y}^{op}(q)}$ and so {\bf (b)} follows from (i). {\bf (c)} is immediate from {\bf (a)} and {\bf (b)}.
\end{proof}
It follows that $\overline{R_Y}^{op} \dashv \overline{L_Y}^{op}$ where on morphism $\overline{L_Y}^{op}$ is defined by,
\begin{definition}
For any dcpo homomorphism $r: \Omega_{RY} W_g \rTo \Omega_{RY} W'_{g'}$ define $\overline{L_Y}^{op}(r)$ to be
\begin{eqnarray*}
\tau^{-1}_{L_Y(W_g),W'_{g'}}(\Omega_{RY} R_Y L_Y (W_g) \rTo^{\Omega_{RY} \eta_W}  \Omega_{RY}W_g \rTo^{r} \Omega_{RY} W'_{g'})\text{.}
\end{eqnarray*}
\end{definition}
That $\overline{L_Y}^{op}$ is functorial and right adjoint to $\overline{R_Y}^{op}$ is immediate from the fact that $\tau$ is an order isomorphism that is natural with respect to dcpo homomorphisms between $\Omega_Y X_f$s. That the resulting adjunction $\overline{L_Y} \dashv \overline{R_Y}$ is a lifting of $L_Y \dashv R_Y$ follows as it is shown above (Lemma \ref{define_tau}) that $\tau_{X_f,W_g} \Omega_Y(n) = \Omega_{RY} \widetilde{n}$. It follows that:
\begin{theorem}\label{slicestability}
The conditions contained in (iii) of the main theorem are stable under slicing: if $L\dashv R:\mathbf{Loc}_{\mathcal{F}}\pile{\rTo \\ \lTo} \mathbf{Loc}_{\mathcal{E}}$ is an order-enriched adjunction with $R$ preserving finitary coproduct and there is a monad isomorphism $\phi:R\mathbb{P}_{\mathcal{E}}\rTo^{\cong} \mathbb{P}_{\mathcal{F}}R$, then for any locale $Y$ over $\mathcal{E}$, $R_Y$ in the sliced adjunction $L_Y\dashv R_Y:\mathbf{Loc}_{\mathcal{F}}/RY \pile{\rTo \\ \lTo} \mathbf{Loc}_{\mathcal{E}}/Y$ preserves finitary coproduct and there is a monad isomorphism  $\phi_Y:R_Y\mathbb{P}_Y \rTo^{\cong} \mathbb{P}_{RY}R_Y$.
\end{theorem}
However for our proof of Proposition \ref{Beck_Chevalley} we are going to need to be more explicit about $\overline{L_Y}^{op}$:
\begin{lemma}\label{L_Lemma}
If $r: \Omega_{RY} W_g \rTo \Omega_{RY} W'_{g'}$ is a dcpo homomorphism relative to $Sh_{\mathcal{F}}(RY)$ then $\gamma^Y_*\overline{L_Y}^{op}(r)= \overline{L}^{op}\gamma_*^{RY}(r)$.
\end{lemma}
\begin{proof}
$\gamma_*^{RY}(r)$ factors as $\Omega_{\mathcal{F}} W \rTo^{\Omega_{\mathcal{F}} \pi_1} \Omega_{\mathcal{F}} ( W \times_{RY} W') \rTo^{a^r} \Omega_{\mathcal{F}} W'$, and $\gamma_*^Y ( \overline{L_Y}^{op} (r))$ factors as $\Omega_{\mathcal{E}} LW \rTo^{\Omega_{\mathcal{E}} \pi_1} \Omega_{\mathcal{E}} (LW \times_Y LW') \rTo^{a^{\overline{L_Y}^{op}(r)}} \Omega_{\mathcal{E}} LW'$ so we need to check that
\begin{eqnarray*}
a^{\overline{L_Y}^{op}(r)}\Omega_{\mathcal{E}}\pi_1 = \overline{L}^{op}(a^r)\Omega_{\mathcal{E}} L \pi_1
\end{eqnarray*}
where $\pi_1: LW \times_Y LW' \rTo LW$ (on the left hand side) and $L \pi_1 : L (W \times_{RY} W') \rTo LW$ (on the right hand side).
Now, by the second part of the naturality lemma (Lemma \ref{Naturality}), the weak triquotient assignment corresponding to the morphism $ \Omega_{RY} R_Y L_Y (W_g) \rTo^{\Omega_{RY} \eta_W} \Omega_{RY} W_g \rTo^{r} \Omega_{RY} W'_{g'}$ used in the definition of $\overline{L_Y}^{op}(r)$ is
\begin{eqnarray*}
\Omega_{\mathcal{F}}( RLW \times_{RY} W') \rTo^{\Omega_{\mathcal{F}} (\eta_W \times Id_{W'})} \Omega_{\mathcal{F}}( W \times_{RY} W') \rTo^{a^r} \Omega_{\mathcal{F}} W' \text{.}
\end{eqnarray*}
Therefore, from the definition of $\tau^{-1}$, we have that $a^{\overline{L_Y}^{op}(r)}$ is equal to the adjoint transpose (via $\overline{R}^{op} \dashv \overline{L}^{op}$) of
\begin{eqnarray*}
\Omega_{\mathcal{F}}( RLW \times_{RY} RLW') \rTo^{\Omega_{\mathcal{F}} (\eta_W \times \eta_{W'})} \Omega_{\mathcal{F}}( W \times_{RY} W') \rTo^{a^r} \Omega_{\mathcal{F}} W'\text{.}
\end{eqnarray*}
Hence $a^{\overline{L_Y}^{op}(r)}$ is equal to the map
\begin{eqnarray*}
\Omega_{\mathcal{E}}(LW \times_Y LW') & \rTo^{\Omega_{\mathcal{E}} \epsilon_{LW \times_Y LW'}} & \Omega_{\mathcal{E}} LR(LW \times_Y LW')\rTo^{\cong} \Omega_{\mathcal{E}}L(RLW \times_{RY} RLW')\\
& \rTo^{\Omega_{\mathcal{E}}L(\eta_W \times \eta_{W'})} & \Omega_{\mathcal{E}} L (W \times_{RY} W') \rTo^{\overline{L}^{op}(a^r)}\Omega_{\mathcal{E}} LW'
\end{eqnarray*}
and the proof reduces to checking that $L \pi_1 : L (W \times_{RY} W') \rTo LW$ factors as $L(W \times_{RY} W') \rTo^{L(\eta_W \times \eta_{W'})} L(RLW \times_{RY} RLW') \rTo^{L(R \pi_1, R \pi_2)^{-1}} LR(LW \times_Y L W') \rTo^{\epsilon_{LW \times_Y LW'}} LW \times_Y LW' \rTo^{\pi_1} LW$; this is clear because (i) $\pi_1 \epsilon_{LW \times_Y LW'} = \epsilon_{LW} LR \pi_1 $ by naturality of $\epsilon$, (ii) $LR \pi_1 L(R \pi_1, R \pi_2 )^{-1} L(\eta_W \times \eta_{W'}) = L \pi_1 L(\eta_W \times \eta_{W'}) = L \eta_W L \pi_1 $ and (iii) $\epsilon_{LW}L  \eta_W = Id_{LW}$.
\end{proof}
\subsection{Proof of Proposition \ref{Beck_Chevalley}}
\begin{proof}
Given Lemma \ref{characterizepreservingstrength} what is required is a proof that  $\overline{R}^{op}(q^Y) \cong [\overline{R}^{op}(q)]^{RY}$ for any dcpo homomorphism $q: \Omega_{\mathcal{E}} X_1 \rTo \Omega_{\mathcal{E}} X_2$. From the explicit description of the natural transformation corresponding to a dcpo homomorphism given in \cite{towdcpo} it can be seen that $q^Y$ is given by $\gamma_*^Y (\gamma^Y)^{\#} (q)$ so by part (ii) of Lemma \ref{R_Y_Functorial} (which shows that $\gamma^{RY}_* \overline{R_Y}^{op}( \_)=  \overline{R}^{op} \gamma^Y_* (\_)$) the proof reduces to checking that $\overline{R_Y}^{op} (\gamma^Y)^{\#} \cong (\gamma^{RY})^{\#} \overline{R}^{op}$ via $\Omega_{\mathcal{F}}(R\pi_1,R\pi_2)$. This follows from Corollary \ref{Rights_are_same} with $\mathcal{D} = \mathbf{Loc}_{\mathcal{F}}/RY$, $\mathcal{C} = \mathbf{Loc}_{\mathcal{E}}$ and $\Sigma_Y L_Y \dashv R_Y Y^*$ and $L\Sigma_{RY} \dashv (RY)^* R$ as the two adjunctions. Here we are of course using the notation $\Sigma_Y : \mathbf{Loc} / Y \rTo \mathbf{Loc}$ for the forgetful functor and $Y^*$ for its right adjoint, the pullback functor. The corollary is applicable since by Lemma \ref{L_Lemma} $\gamma^Y_*\overline{L_Y}^{op}(r)= \overline{L}^{op}\gamma_*^{RY}(r)$; i.e. $\overline{\Sigma_Y}\overline{L_Y} = \overline{L} \overline{\Sigma_{RY}}$, and $(RY)^*R \cong R_Y Y^*$ via $(R\pi_1,R\pi_2): R(\_ \times \_ ) \rTo ^{\cong} R(\_) \times R(\_)$.
\end{proof}
Note that Beck-Chevalley holds for the diagram of adjunctions,
\begin{diagram}
{\mathbf{Loc}_{\mathcal{F}}}/RY & \pile{\rTo^{L_Y} \\ \lTo_{R_Y}} & {\mathbf{Loc}_{\mathcal{E}}}/Y \\
\dTo^{\Sigma_{RY}} \dashv \uTo_{(RY)^*} &   &  \dTo^{\Sigma_Y} \dashv \uTo_{Y^*} \\
\mathbf{Loc}_{\mathcal{F}} & \pile{\rTo^L \\ \lTo_R} & \mathbf{Loc}_{\mathcal{E}} \\
\end{diagram}
The proof just given shows that Beck-Chevalley also holds for the lifting to Kleisli categories of this diagram of adjunctions. Indeed, conversely, if Beck-Chevalley holds for the lifting of the above diagram then $R$ preserves the strength.
\section{Representing geometric morphisms using the lower and upper power locale monads}
This section proves the equivalence of (iii) and (iv) of the main theorem. See, for example, \cite{PrePrePre} for the definitions of the lower and the upper power locale monads, $P_L$ and $P_U$ respectively. The key property of their functor parts is that for each locale $X$ there is a universal suplattice (respectively preframe) homomorphism $\Omega X \rTo^{\diamond_{X}} \Omega P_L (X)$ ($\Omega X \rTo^{\square_{X}} \Omega P_U (X)$) such that for every suplattice (respectively preframe) homomorphism $\rho: \Omega X \rTo  \Omega Y$ there is a unique map $f_{\rho}: Y \rTo P_L(X)$ (respectively $f_{\rho}: Y \rTo P_U(X)$) such that $\rho=\Omega f_{\rho}\diamond_X$ ($\rho=\Omega f_{\rho}\square_X$). The rest of the monad structure follows from this defining universal property of the functor $P_L:\mathbf{Loc}\rTo\mathbf{Loc}$ ($P_U:\mathbf{Loc}\rTo\mathbf{Loc}$). The resulting Kleisli category $\mathbf{Loc}_{P_L}$ ($\mathbf{Loc}_{P_U}$) is isomorphic to the opposite of the category whose objects are frames and whose morphisms are suplattice (preframe) homomorphisms.
\begin{definition}
If $\mathcal{C}$ is an order-enriched category with two monads $\mathbb{T}%
_{a}\mathbb{=}(T_{a},\eta ^{a},\mu ^{a})$ and $\mathbb{T}_{a}\mathbb{=}%
(T_{b},\eta ^{b},\mu ^{b})$ then a \emph{distribution isomorphism} is a natural isomorphism $\psi :T_{a}T_{b}\rTo T_{b}T_{a}$
such that the diagrams
\begin{diagram}
&& T_{a} &&  \\
&\ldTo^{T_{a}\eta ^{b}} &  & \rdTo^{\eta _{T_{a}}^{b}}& \\
T_{a}T_{b} && \rTo^{\psi } && T_{b}T_{a} \\
&\luTo^{\eta _{T_{b}}^{a}} &  & \ruTo^{T_{b}\eta ^{a}} &\\
&& T_{b} &&
\end{diagram}
and
\begin{diagram}
T_{a}T_{a}T_{b} & \rTo^{T_{a}\psi }& T_{a}T_{b}T_{a} &
\rTo^{\psi _{T_{a}}} & T_{b}T_{a}T_{a} \\
\dTo^{\mu _{T_{b}}^{a}} &  &  &  & \dTo^{T_{b}\mu ^{a}} \\
T_{a}T_{b} &  & \rTo^{\psi } &  & T_{b}T_{a} \\
\uTo^{T_{a}\mu ^{b}}&  &  &  & \uTo^{ \mu _{T_{a}}^{b}} \\
T_{a}T_{b}T_{b} & \rTo^{\psi _{T_{b}}} & T_{b}T_{a}T_{b}
& \rTo^{T_{b}\psi } & T_{b}T_{b}T_{a}
\end{diagram}
both commute.
\end{definition}
The notion of a distribution isomorphism is of interest because given two monads and a distribution isomorphism a third monad naturally arises:
\begin{lemma}
\label{doublelemma}If $\mathcal{C}$ is an order-enriched category with
two order-enriched monads $\mathbb{T}_{a}$ and $\mathbb{T}_{b}$ for which there is a distribution isomorphism $\psi: T_a T_b \rTo^{\cong} T_b T_a$
then there is another order-enriched monad, $\mathbb{T}_{ab}$, given by
\begin{equation*}
(T_{a}T_{b},Id\overset{\eta ^{a}}{\rTo }T_{a}\overset{T_{a}\eta
^{b}}{\rTo }T_{a}T_{b},T_{a}T_{b}T_{a}T_{b}\overset{T_{a}\psi^{-1}
_{T_{b}}}{\rTo }T_{a}T_{a}T_{b}T_{b}\overset{\mu ^{a}_{T_b T_b}}{%
\rTo }T_{a}T_{b}T_{b}\overset{T_a \mu^b }{\rTo }%
T_{a}T_{b})\text{.}
\end{equation*}
\end{lemma}
\begin{proof}
The proof is a routine though lengthy diagram chase.
\end{proof}
The paper \cite{PrePrePre} provides an example of a distribution isomorphism by effectively showing that there is a distribution isomorphism between the lower and upper power locale monads. Indeed it was the observation that the lower and upper power locales commute and so give rise to a third power locale (the `double' power locale) which initially attracted interest in the study of $\mathbb{P}$. The distribution isomorphism, $\psi_X : P_L P_U (X)  \rTo^{\cong} P_U P_L (X)$, is the unique locale map such that map $\Omega \psi_X \square_{P_LX} \diamond_X = \diamond_{P_UX} \square_X$. Note, (\cite{PrePrePre}), that  both $\square_{P_LX} \diamond_X$ and $\diamond_{P_UX} \square_X$ are universal dcpo homomorphism (with codomains $\Omega P_L P_U (X)$ and $\Omega P_U P_L (X)$ respectively).

Our final definition provides clarity on the terms used in the statement (iv) of the main theorem:
\begin{definition}\label{preserves_distribution}
 If $\mathcal{C}$ and $\mathcal{D}$ are two order-enriched categories each with two monads; $\mathbb{T}^{\mathcal{C}}_a,$ $\mathbb{T}^{\mathcal{C}}_b$, $\mathbb{T}^{\mathcal{D}}_a$, and $\mathbb{T}^{\mathcal{D}}_b$, for which there are two distribution isomorphisms, $\psi^{\mathcal{C}} :T^{\mathcal{C}}_{a}T^{\mathcal{C}}_{b}\rTo T^{\mathcal{C}}_{b}T^{\mathcal{C}}_{a}$ and $\psi^{\mathcal{D}} :T^{\mathcal{D}}_{a}T^{\mathcal{D}}_{b}\rTo T^{\mathcal{D}}_{b}T^{\mathcal{D}}_{a}$, then for any $F: \mathcal{C} \rTo \mathcal{D}$ for which there are two monad opfunctors $ \phi^a: F T_a^{\mathcal{C}} \rTo T_a^{\mathcal{D}} F$ and $ \phi^b: F T_b^{\mathcal{C}} \rTo T_b^{\mathcal{D}} F$ then $F$ is said to \emph{preserve the distribution isomorphism up to the monad opfunctors} provided the diagram
\begin{diagram}
F_a^{\mathcal{C}}  T_b^{\mathcal{C}}& \rTo^{F \psi^{\mathcal{C}}} & F T^{\mathcal{C}}_b T_a^{\mathcal{C}} \\
\dTo^{\phi^a_{T_b^{\mathcal{C}}}}   &                             & \dTo^{\phi^b_{T_a^{\mathcal{C}}}}    \\
T_a^{\mathcal{D}} F T_b^{\mathcal{C}} &                           & T_b^{\mathcal{D}} F T_a^{\mathcal{C}} \\
\dTo^{T_a^{\mathcal{D}} \phi^b }      &                            & \dTo^{T_b^{\mathcal{D}}\phi^a}         \\
T_a^{\mathcal{D}} T_b^{\mathcal{D}} F & \rTo^{\psi_F^{\mathcal{D}}} & T_b^{\mathcal{D}} T_a^{\mathcal{D}} F \\
\end{diagram}
commutes.
\end{definition}
With this we can now establish (iii) implies (iv) of the main theorem: Lemma \ref{suplatticespeciallemma} shows how to lift $L \dashv R$ to the Kleisli categories of $P_L$ and since the proof of that lemma works equally well with finitary frame meet in place of join it is also clear how to lift $L \dashv R$ to the Kleisli categories of $P_U$. This results in two monad isomorphisms $\phi^L: R P^{\mathcal{E}}_L \rTo^{\cong} P^{\mathcal{F}}_L R$ (which, for each locale $X$ over $\mathcal{E}$ corresponds to the unique frame homomorphism $\Omega_{\mathcal{F}} \phi^L_X$ such that $\Omega_{\mathcal{F}} \phi^L_X \diamond_{RX} = \overline{R}^{op}(\diamond_X)$) and $\phi^U: R P^{\mathcal{E}}_U \rTo^{\cong} P^{\mathcal{F}}_U R$ (such that $\Omega_{\mathcal{F}} \phi^U_X \square_{RX} = \overline{R}^{op}(\square_X)$ for each $X$). So to complete a verification of (iii) implies (iv) of the main theorem it remains to check that $R$ preserves $\psi$ (up to $\phi^L$ and $\phi^U$). Because $\square_{P^{\mathcal{F}}_L(RX)} \diamond_{RX}$ is a universal dcpo homomorphism (from $\Omega_{\mathcal{F}} (RX)$ to $\Omega_{\mathcal{F}} P_U^{\mathcal{F}}P^{\mathcal{F}}_L(RX)$) to prove that $R$ preserves $\psi$ it is sufficient to check, for each locale $X$ over $\mathcal{E}$, that
\begin{eqnarray*}
\Omega_{\mathcal{F}}(R \psi^{\mathcal{E}}_X)\Omega_{\mathcal{F}}(\phi^U_{P^{\mathcal{E}}_L X})\Omega_{\mathcal{F}}(P^{\mathcal{F}}_U \phi^L_X)\square_{P^{\mathcal{F}}_L(RX)} \diamond_{RX} \text{ \ \ \ \ \ \ \ \ (I)}
\end{eqnarray*}
is equal to
\begin{eqnarray*}
 \Omega_{\mathcal{F}}( \phi^L_{P^{\mathcal{E}}_U X}) \Omega_{\mathcal{F}}(P^{\mathcal{F}}_L \phi^U_X) \Omega_{\mathcal{F}}(\psi^{\mathcal{F}}_{RX}) \square_{P^{\mathcal{F}}_L(RX)} \diamond_{RX} \text{. \ \ \ \ \ \ \ \ (II)}
\end{eqnarray*}
(I) is equal to
\begin{eqnarray*}
 && \Omega_{\mathcal{F}}(R \psi^{\mathcal{E}}_X)\Omega_{\mathcal{F}}(\phi^U_{P^{\mathcal{E}}_L X})\square_{R P^{\mathcal{E}}_L(X)} \Omega_{\mathcal{F}}( \phi^L_X)\diamond_{RX} \\
  & = & \Omega_{\mathcal{F}}(R \psi_X^{\mathcal{E}}) \overline{R}^{op}(\square_{P^{\mathcal{E}}_L(X)}) \overline{R}^{op} ( \diamond_X) \\
  & = & \overline{R}^{op}[\Omega_{\mathcal{E}}( \psi_X^{\mathcal{E}}) \square_{P^{\mathcal{E}}_L(X)}  \diamond_X] \\
  & = & \overline{R}^{op}(\diamond_{P_U^{\mathcal{E}}(X)}\square_X) \text{.}
\end{eqnarray*}
(II) on the other hand is equal to
\begin{eqnarray*}
&& \Omega_{\mathcal{F}}( \phi^L_{P^{\mathcal{E}}_U X}) \Omega_{\mathcal{F}}(P^{\mathcal{F}}_L \phi^U_X) \diamond_{P^{\mathcal{F}}_U(RX)} \square_{RX} \\
& = & \Omega_{\mathcal{F}}( \phi^L_{P^{\mathcal{E}}_U X})  \diamond_{RP^{\mathcal{E}}_U(X)} \Omega_{\mathcal{F}}(\phi^U_X)\square_{RX} \\
& = & \overline{R}^{op}(\diamond_{P_U^{\mathcal{E}}(X)}) \overline{R}^{op}(\square_X) \\
& = &  \overline{R}^{op}(\diamond_{P_U^{\mathcal{E}}(X)}\square_X)\text{.}
\end{eqnarray*}
This completes our verification of (iii) implies (iv) of the main theorem.

In the other direction, given an order-enriched adjunction $L\dashv R$ with $\phi_L:R P_L^{\mathcal{E}}\rTo^{\cong} P_L^{\mathcal{F}}R$ and $\phi_U:R P_U^{\mathcal{E}}\rTo^{\cong} P_U^{\mathcal{F}}R$ such that $R$ preserves $\psi: P_L P_U \rTo^{\cong} P_U P_L$ then to complete a proof of (iv) implies (iii) two final facts need to be checked: (a) there is monad isomorphism $\phi : R \mathbb{P}_{\mathcal{E}}  \rTo^{\cong} \mathbb{P}_{\mathcal{F}} R$ and (b) $R$ preserves finitary coproduct. For (a) note that because $\mathbb{P}$ arises from the lower and upper power locale monads (as in Lemma \ref{doublelemma}) we can rely on a general proof that if there are two monad isomorphisms (and $R$ preserves the given distribution isomorphism) then there is a monad isomorphism on the composite monad and for this it is sufficient to show:
\begin{lemma}
Given a distribution isomorphism with $F$ preserving it as in Definition \ref{preserves_distribution}, then $\phi:F T^{\mathcal{C}}_{ab} \rTo  T^{\mathcal{D}}_{ab} F$
given by $F T^{\mathcal{C}}_a T^{\mathcal{C}}_b \rTo^{\phi^a_{T^{\mathcal{C}}_b}} T^{\mathcal{D}}_a F T^{\mathcal{C}}_b \rTo^{T^{\mathcal{D}}_a \phi^b} T^{\mathcal{D}}_a T^{\mathcal{D}}_b F$ is a monad opfunctor.
\end{lemma}
\begin{proof}
The proof is a routine diagram chase. For the first diagram in the definition of a monad opfunctor, note that
\begin{diagram}
                            &                    &                             &                                  &                F  \\
                            &                    &                             & \ldTo^{F\eta^{\mathcal{C},a}}  & \dTo_{\eta^{\mathcal{D},a}_F} \\
                            &                    &  FT^{\mathcal{C}}_a          & \rTo^{\phi^a}                    & T^{\mathcal{D}}_a F \\
                            &\ldTo^{F T^{\mathcal{C}}_a\eta^{\mathcal{C},b}}&  & \ldTo^{ T^{\mathcal{D}}_a F \eta^{\mathcal{C},b}} & \dTo_{T^{\mathcal{D}} \eta^{\mathcal{D},b}_F} \\
FT^{\mathcal{C}}_a T^{\mathcal{C}}_b & \rTo^{\phi^a_{T^{\mathcal{C}}_b}} & T^{\mathcal{D}}_a F T^{\mathcal{C}}_b & \rTo^{T^{\mathcal{D}}_a \phi^b} & T^{\mathcal{D}}_a T^{\mathcal{D}}_b F \\
\end{diagram}
commutes because $\phi^a$ and $\phi^b$ are monad opfunctors. Finally notice that the diagram
\begin{diagram}
FT^{\mathcal{C}}_{abab} & \rTo{\phi^a_{T^{\mathcal{C}}_{bab}}} & T^{\mathcal{D}}_a F T^{\mathcal{C}}_{bab} & \rTo^{T^{\mathcal{D}}_a \phi^b_{T^{\mathcal{C}}_{ab}}} & T^{\mathcal{D}}_{ab}  F T^{\mathcal{C}}_{ab} & \rTo^{T^{\mathcal{D}}_{ab} \phi^a_{T^{\mathcal{C}}_b}} & T^{\mathcal{D}}_{aba} F T^{\mathcal{C}}_b & \rTo^{T^{\mathcal{D}}_{aba} \phi^b} & T^{\mathcal{D}}_{abab} F \\
\dTo_{FT^{\mathcal{C}}_a (\psi^{\mathcal{C}})^{-1}_{T^{\mathcal{C}}_b}} & & \dTo_{T^{\mathcal{D}}_a F(\psi^{\mathcal{C}})^{-1}_{T^{\mathcal{C}}_b}} & &  & & \dTo^{T^{\mathcal{D}}_a (\psi^{\mathcal{D}})^{-1}_{F T^{\mathcal{C}}_b}} &  & \dTo^{T^{\mathcal{D}}_a ( \psi^{\mathcal{D}})^{-1}_{T^{\mathcal{D}}_b F}} \\
FT^{\mathcal{C}}_{aabb} & \rTo^{\phi^a_{T^{\mathcal{C}}_{abb}}} & T^{\mathcal{D}}_a F T^{\mathcal{C}}_{abb} & \rTo^{T^{\mathcal{D}}_a \phi^a_{T^{\mathcal{C}}_{bb}}} & T^{\mathcal{D}}_{aa} F T^{\mathcal{C}}_{bb} & \rTo^{T^{\mathcal{D}}_{aa} \phi^b_{T^{\mathcal{C}}_b}} & T^{\mathcal{D}}_{aab} F T^{\mathcal{C}}_b & \rTo^{T^{\mathcal{D}}_{aab} \phi^b } & T^{\mathcal{D}}_{aabb} F \\
\dTo_{F \mu^{\mathcal{C},a}_{T^{\mathcal{C}}_{bb}}} & & & & \dTo^{\mu^{\mathcal{D},a}_{F T^{\mathcal{C}}_{bb}}} & & \dTo^{\mu^{\mathcal{D},a}_{T^{\mathcal{D}}_b F T^{\mathcal{C}}_b}} & & \dTo^{\mu^{\mathcal{D},a}_{T^{\mathcal{D}}_{bb} F}} \\
FT^{\mathcal{C}}_{abb} &  & \rTo^{\phi^a_{T^{\mathcal{C}}_{bb}}}   & & T^{\mathcal{D}}_a F T^{\mathcal{C}}_{bb} & \rTo^{T^{\mathcal{D}}_a \phi^b_{T^{\mathcal{C}}_b}} & T^{\mathcal{D}}_{ab} F T^{\mathcal{C}}_b & \rTo^{T^{\mathcal{D}}_{ab} \phi^b }& T^{\mathcal{D}}_{abb}  F \\
\dTo_{F T^{\mathcal{C}}_a \mu^{\mathcal{C},b}} & & & & \dTo^{T^{\mathcal{D}}_a F \mu^{\mathcal{C},b}} & & & & \dTo^{T^{\mathcal{D}}_a \mu^{\mathcal{D},b}_F} \\
FT^{\mathcal{C}}_{ab} &  &  \rTo^{\phi^a_{T^{\mathcal{C}}_b}} &   &  T^{\mathcal{D}}_a F T^{\mathcal{C}}_b &  & \rTo^{T^{\mathcal{D}}_a \phi^b}  &  & T^{\mathcal{D}}_{ab} F \\
\end{diagram}
commutes (where we are following the notation $T_{abc} = T_a T_b T_c$ etc). To see this note that the rectangle in the first row commutes as $F$ preserves the distribution isomorphism, the rectangle in the middle row commutes and the bottom right rectangle commutes because $\phi^a$ and $\phi^b$ respectively are monad opfunctors. The other squares and rectangle commute by naturality.
\end{proof}
\begin{lemma}
Given an order-enriched adjunction $L\dashv R:\mathbf{Loc}_{\mathcal{F}}\pile{\rTo \\ \lTo} \mathbf{Loc}_{\mathcal{E}}$ with a monad isomorphism $\phi_L:R P_L^{\mathcal{E}}\cong P_L^{\mathcal{F}}R$, $R$ preserves finitary coproduct.
\end{lemma}
\begin{proof}
Because there is a monad isomorphism there is a lifted adjunction $\overline{R}^{op} \dashv \overline{L}^{op}$ between categories whose objects are frames and whose morphisms are suplattice homomorphisms (i.e. between the opposites of the Kleisli categories of $P^{\mathcal{E}}_L$ and $P^{\mathcal{F}}_L$).
For any locale $W$ over $\mathcal{F}$,
\begin{eqnarray*}
\mathbf{Sup}_{\mathcal{F}}( \Omega_{\mathcal{F}}R 0_{\mathcal{E}},\Omega_{\mathcal{F}}W) & \cong & \mathbf{Sup}_{\mathcal{E}}(\Omega_{\mathcal{E}}0_{\mathcal{E}} , \Omega_{\mathcal{E}}LW) \\
& \cong & \{*\} \\
& \cong & \mathbf{Sup}_{\mathcal{F}}(\Omega_{\mathcal{F}}0_{\mathcal{F}} , \Omega_{\mathcal{F}}W) \\
\end{eqnarray*}
naturally in suplattice homomorphisms between $\Omega_{\mathcal{F}}W$s. It follows that $R0_{\mathcal{E}} \cong 0_{\mathcal{F}}$.

For binary coproduct say $X_1$ and $X_2$ are locales over $\mathcal{E}$. Then,
\begin{eqnarray*}
 \mathbf{Sup}_{\mathcal{F}}(\Omega_{\mathcal{F}}R(X_1+X_2) , \Omega_{\mathcal{F}}W) & \cong & \mathbf{Sup}_{\mathcal{E}}(\Omega_{\mathcal{E}}(X_1 + X_2) , \Omega_{\mathcal{E}}LW) ) \\
 & \cong & \mathbf{Sup}_{\mathcal{E}}(\Omega_{\mathcal{E}}X_1 , \Omega_{\mathcal{E}}LW) \times   \mathbf{Sup}_{\mathcal{E}}(\Omega_{\mathcal{E}}X_2 , \Omega_{\mathcal{E}}LW)\\
 & \cong & \mathbf{Sup}_{\mathcal{F}}(\Omega_{\mathcal{F}}RX_1 , \Omega_{\mathcal{F}}W) \times   \mathbf{Sup}_{\mathcal{F}}(\Omega_{\mathcal{F}}RX_2 , \Omega_{\mathcal{F}}W)\\
\end{eqnarray*}
naturally in suplattice homomorphisms between $\Omega_{\mathcal{F}}W$s. It follows that $\Omega_{\mathcal{F}}R(X_1 + X_2)$ is the suplattice coproduct of $\Omega_{\mathcal{F}}R X_1$ and $\Omega_{\mathcal{F}}R X_1$ and so $R(X_1+X_2) \cong RX_1 + RX_2$ since suplattice coproduct is given by set theoretic product.

Intuitively this completes the proof, however to properly complete the proof we do need to also check that it is indeed the canonical map $RX_1+RX_2 \rTo^{ [R \coprod_1 , R \coprod_2 ] } R(X_1 +X_2)$ that is the isomorphism in question. Firstly note that by looking at the image of the identity $Id_{\Omega_{\mathcal{F}} R(X_1 + X_2)}$ under the order isomorphisms above it can be seen that $\overline{R}^{op} \coprod^S_i : \Omega_{\mathcal{F}} RX_i \rTo \Omega_{\mathcal{F}} R(X_1 + X_2)$ for $i=1,2$ are the suplattice coprojections to $\Omega_{\mathcal{F}} R(X_1 + X_2)$ where $\coprod^S_i : \Omega_{\mathcal{E}} X_i \rTo \Omega_{\mathcal{E}} (X_1 + X_2)$ are the suplattice coprojections (so, $\coprod^S_1(a) = (a, 0_{\Omega X_2})$ and  $\coprod^S_2(b) = (0_{\Omega X_1}, b)$). Notice that since $\Omega_{\mathcal{E}} \coprod_i \coprod^S_i = Id$ for $i=1,2$ and $\Omega_{\mathcal{E}} \coprod_i \coprod^S_j = 0$ for $i \neq j$ we have that $\Omega_{\mathcal{F}} R \coprod_i \overline{R}^{op}\coprod^S_i = Id$ for $i=1,2$ and $\Omega_{\mathcal{F}} R \coprod_i \overline{R}^{op} \coprod^S_j = 0$ for $i \neq j$; the last because $\overline{R}^{op}$ preserves the zero map since it is left adjoint to the unique map back to $\Omega 0$ and we have established that $R$ preserves the zero object. If we then define a suplattice homomorphism $\psi : \Omega_{\mathcal{F}} (RX_1+RX_2)  \rTo  \Omega_{\mathcal{F}} R(X_1 + X_2)$ by $\psi(a,b) = \overline{R}^{op} \coprod^S_1 a \vee \overline{R}^{op} \coprod^S_2 b $ then certainly $\Omega_{\mathcal{F}}[R\coprod_1,R\coprod_2]\psi = Id$ because $\Omega_{\mathcal{F}}[R\coprod_1,R\coprod_2](c) =(\Omega_{\mathcal{F}}R\coprod_1(c),\Omega_{\mathcal{F}}R\coprod_2(c))$. But to prove that $\psi \Omega_{\mathcal{F}} [ R \coprod_1, R \coprod_2] = Id$ it is sufficient to prove that $\psi \Omega_{\mathcal{F}} [ R \coprod_1, R \coprod_2]\overline{R}^{op} \coprod^S_i  =\overline{R}^{op} \coprod^S_i  $ for $i=1,2$ since we have noted that $\overline{R}^{op} \coprod^S_i$ are the suplattice coprojections. This last is immediate from the properties $ \Omega_{\mathcal{E}}R \coprod_i$ and $ \overline{R}^{op} \coprod^S_j$ already noted and so $[R \coprod_1 , R \coprod_2 ]$ is an isomorphism and the proof is complete.
\end{proof}
\section{Final comments}
For this paper, applying the techniques of \cite{towgeom}, we have established that geometric morphisms can be represented as those adjunctions between the corresponding categories of locales that commute with the double power locale monad and for which the right adjoint preserves finite coproduct. Additionally it has been shown that geometric morphisms correspond to adjunctions that commute with the lower and upper power locale monads and for which the right adjoint preserves the distribution between the lower and upper power locales. Along the way we have, in effect, made the technical observation that the key to proving the representation of geometric morphisms in this manner is knowing that the right adjoint also preserves the strength of the double power locale monad.

In fact the power constructions can be developed axiomatically (e.g. \cite{Vicpoints} and \cite{towhofman}, though for this last see \cite{victow} for the driving observation which shows how to interpret the double power construction as an exponential). It is therefore possible to develop a categorical account of geometric morphisms. A `categorical' geometric morphism is an adjunction that commutes with the power locale monads. In such a context it is clear what a localic geometric morphism should be and, further, if one defines a hyperconnected geometric morphism to be corresponding to an adjunction such that $L1\cong1$ it is possible to prove that every (categorical interpretation of) a geometric morphism factors uniquely as a hyperconnected geometric morphism followed by a localic one. The proof is essentially straightforward and relies on a categorical proof of the slice stability result above (Theorem \ref{slicestability}) available since there is a categorical account of weak triquotient assignments in locale theory, see \cite{towaxioms}. However, the situation would be much more appealing if one could show that, categorically, the property of commuting with the power locale monads implies that Frobenius reciprocity holds of the pullback adjunction. This is true for categories of locales (\cite{towgeom}) and were it to be true categorically then the categorical account of the hyperconnected-localic factorization would be pullback stable. In other words this extra step would mean that the categorical situation is more in keeping with what we know to be the case from topos theory.
In summary, whilst the main result of this paper shows that the notion of geometric morphism is closely related to the power locale constructions, to develop the theory further it appears that more work is required on categorical interpretations of locale theory.


\begin{thebibliography}{test}

\bibitem[J02]{Elephant}  Johnstone, P.T. \emph{Sketches of an elephant: A
topos theory compendium}. Vols 1, 2, Oxford Logic Guides \textbf{43},
\textbf{44}, Oxford Science Publications, 2002.

\bibitem[JV91]{PrePrePre}  Johnstone, P.T., and Vickers, S.J. ``Preframe
presentations present'', in Carboni, Pedicchio and Rosolini (eds) Category
Theory -- Proceedings, Como, 1990 (\emph{Springer Lecture Notes in
Mathematics} \textbf{1488}, 1991), 193-212.

\bibitem[JT84]{JoyT}  Joyal, A. and Tierney, M. \emph{An Extension of the Galois
Theory of Grothendieck}, Memoirs of the American Mathematical Society
\textbf{309}, 1984.

\bibitem[P97]{triquot}  Plewe, T. \emph{Localic triquotient maps are
effective descent maps}, Math. Proc. Cambridge Philos. Soc. \textbf{122}
(1997), 17-43.

\bibitem[P70]{Pumplun}  Pumpl$\overset{\text{..}}{\text{u}}$n, D. \emph{Eine
Bemerkung }$\overset{\text{\emph{..}}}{\emph{u}}$\emph{ber Monaden und
adjungierte Funktoren}. Math. Ann. \textbf{185} (1970). 329-337.\emph{\ }

\bibitem[S72]{Ross}  Street, R. \emph{The formal theory of monads}. J. Pure
Appl. Algebra \textbf{2} (1972) 149-168

\bibitem[T03]{towdcpo}  Townsend, C.F. \emph{Locale Pullback via Directed
Complete Partial Orders}. Theoretical Computer Science \textbf{316 }(2004)
225-258.

\bibitem[T05]{towhofman}  Townsend, C.F. \emph{A categorical account of the
Hofmann-Mislove theorem. }Math. Proc. Camb. Phil. Soc. \textbf{139} (2005)
441-456.

\bibitem[T10a]{towaxioms}  Townsend, C.F. \emph{An Axiomatic account of Weak Localic Triquotient Assignments.} Journal of Pure and Applied Algebra
Volume 214 \textbf{6} (2010) 729-739.

\bibitem[T10b]{towgeom} Townsend, C.F. \emph{A representation theorem for geometric morphisms.} Applied Categorical Structures. \textbf{18} (2010) 573-583

\bibitem[V95]{Vicpoints} Vickers, S.J. \emph{Locales are not pointless} in `Theory and Formal Methods 1994'. Second Imperial College Department of Computing Workshop on Theory and Formal Methods. Cambridge. 1994. Imperical College Press, London, (1995) 199-216.

\bibitem[V02]{DoubPt}  Vickers, S.J. \emph{The double powerlocale and
exponentiation: a case study in geometric logic}. Theory and Applications of
Categories \textbf{12} (2004) 372-442.

\bibitem[VT04]{victow}  Vickers, S.J. and Townsend, C.F. \emph{A Universal
Characterization of the Double Power Locale}. Theoretical Computer Science
\textbf{316 }(2004) 297-321.
\end{thebibliography}
\end{document}